\documentclass[a4paper, 12pt, final]{article}

\usepackage[scale=0.768]{geometry}
\usepackage[english]{babel}

\usepackage[mathscr]{euscript}
\usepackage{color}
\usepackage{fancyhdr}
\usepackage{dsfont}
\usepackage{bm}
\usepackage{showkeys}

\usepackage{amsfonts}
\usepackage{amsmath}
\usepackage{amssymb}
\usepackage{amsthm}

\usepackage{dblaccnt}
\usepackage{accents}
\usepackage{graphicx}
\usepackage{array}

\newtheorem{theorem}{Theorem}

\newtheorem{lemma}[theorem]{Lemma}

\newtheorem{remark}[theorem]{Remark}

\newcommand*{\N}{\ensuremath{\mathbb{N}}}
\newcommand*{\Z}{\ensuremath{\mathbb{Z}}}
\newcommand*{\R}{\ensuremath{\mathbb{R}}}
\newcommand*{\C}{\ensuremath{\mathbb{C}}}
\renewcommand{\i}{\mathrm{i}}
\renewcommand{\phi}{\varphi}
\renewcommand{\rho}{{\varrho}}
\renewcommand{\epsilon}{{\varepsilon}}

\renewcommand{\d}[1]{\,\mathrm{d}#1 \,}

\newcommand{\F}{\mathcal{F}} 
\newcommand{\J}{\mathcal{J}} 
\newcommand{\0}{{0}} 
\newcommand{\pert}{{\mathrm{p}}}
\newcommand{\T}{{\mathrm{T}}}
\newcommand{\B}{{\mathrm{B}}}

\renewcommand{\tilde}{\widetilde}
\newcommand{\Null}{0}

\newcommand{\W}{{W_{\hspace*{-1pt}\Lambda}}} 
\newcommand{\Wast}{{W_{\hspace*{-1pt}\Lambda^\ast}}} 

\newlength{\dhatheight}
\newcommand{\dhat}[1]{%
    \settoheight{\dhatheight}{\ensuremath{\hat{#1}}}%
    \addtolength{\dhatheight}{-0.35ex}%
    \hat{\vphantom{\rule{1pt}{\dhatheight}}%
    \smash{\hat{#1}}}}

\begin{document}

\sloppy

\title{A Floquet-Bloch Transform Based Numerical Method for Scattering from Locally Perturbed Periodic Surfaces}
\author{Armin Lechleiter\thanks{Center for Industrial Mathematics, University of Bremen
; \texttt{lechleiter@math.uni-bremen.de}} \and 
Ruming Zhang\thanks{Center for Industrial Mathematics, University of Bremen
; \texttt{rzhang@uni-bremen.de}}\ \thanks{corresponding author}}
\date{}
\maketitle

\begin{abstract}
Scattering problems for periodic structures have been studied a lot in the past few years. A main idea for numerical solution methods is to reduce such problems to one periodicity cell. In contrast to periodic settings, scattering from locally perturbed periodic surfaces is way more challenging. In this paper, we introduce and analyze a new numerical method to simulate scattering from locally perturbed periodic structures based on the Bloch transform. As this transform is applied only in periodic domains, we firstly rewrite the scattering problem artificially in a periodic domain. With the help of the Bloch transform, we secondly transform this problem into a coupled family of quasiperiodic problems posed in the periodicity cell. A numerical scheme then approximates the family of quasiperiodic solutions (we rely on the finite element method) and backtransformation provides the solution to the original scattering problem. In this paper, we give convergence analysis and error bounds for a Galerkin discretization in the spatial and the quasiperiodicity's unit cells. We also provide a simple and efficient way of implementation that does not require numerical integration in the quasiperiodicity, together with numerical examples for scattering from locally perturbed periodic surfaces computed by this scheme.  
\end{abstract}

\section{Introduction}
In this paper, we present a numerical method for solving scattering problems from locally perturbed periodic surfaces. Scattering problems for periodic or quasiperiodic incident fields from periodic structures have been well studied over at least 25 years. The common way of solving is reduction to one periodicity cell, which avoids the need for computing numerical solutions in unbounded domains. However, if such reduction fails due to non-periodicity of the incident field or the surface, one needs to seek for new approaches.

The approach we present in this paper is based on the Floquet-Bloch transform. It builds up a relationship between a non-periodic problem and a family of quasiperiodic problems reduced to one single period. With this transform, the scattering problems from periodic surfaces and non-periodic incident fields have been discussed in \cite{Lechl2015e} and \cite{Lechl2016}. Based on these theoretic results, a numerical scheme has been developed to solve these kinds of scattering problems in~\cite{Lechl2016a}. Following this type of technique, we introduce in this paper an algorithm for solving scattering problems from local perturbations of periodic surfaces that is pretty close to the one from the recent paper~\cite{Hadda2016}. Our convergence analysis is for various reasons different, as~\cite{Hadda2016} for instance strongly relies on integral equations in the spatial variable. The source of inspiration for all these techniques seems to be the paper~\cite{Coatl2012} on wave propagation in full-space periodic media. 

To briefly present our numerical approach in some detail, we firstly rely on the Floquet-Bloch transform, defined on functions living in periodic domains, and hence reformulate the locally perturbed problem by a suitable diffeomorphism between the locally perturbed and the purely periodic domain. Applying the Bloch transform to the new problem yields a family of quasiperiodic scattering problems posed in one single periodicity cell. We state the classic error analysis for finite element discretizations using low-order approximation in $\alpha$. The interesting feature of this discretization is that all integrals in the quasiperiodicity parameters can be computed by hand, such that standard solvers become attractive to tackle the full problem. By finite element discretizations for the spatial parts of the problem and, roughly, the trapezoidal rule discretizing the inverse Bloch transform, one gets a large but sparse block-linear system to solve. To this end, we use the GMRES iteration with a specially designed incomplete $LU$-decomposition as pre-conditioner for the numerical solution of the linear system. 

For Dirichlet scattering problems on perturbed periodic surfaces one can, at least in two dimensions, of course exploit the corresponding numerical convergence theory for boundary integral equation approximations from rough surface scattering theory, see, e.g.,~\cite{Meier2000, Arens2002}. There are, however, few methods specifically designed for such locally perturbed periodic scatterers. In \cite{Joly2006, Fliss2009a} and \cite{Fliss2015}, the authors give a method that approximates the Dirichlet-to-Neumann map on the transparent edges of a periodic waveguide modeled by the Helmholtz equation. Another method that uses the so-called recursive doubling procedure constructs the Sommerfeld-to Sommerfeld maps at artificial boundaries of such a waveguide, see \cite{Ehrhardt2009} and \cite{Ehrhardt2009a}. Both of these methods are motivated by the infinite half-guide and inspired by the limiting absorption principle.

This paper is organized as follows. In Section~\ref{se:scatter}, we describe the direct scattering problem corresponding to a locally perturbed periodic surface. In Sections~\ref{se:PerSca} and~\ref{se:qpScattPerturb}, we use the Bloch transform to obtain an equivalent family of quasiperiodic problems. In Section~\ref{se:invBloch} we give a discrete inverse Bloch transform and estimate the finite element method applied to the individual quasiperiodic scattering problems. The numerical implementation for the Bloch transform based method is illustrated in Section~\ref{se:NumMet}. In the last Section~\ref{se:NumEg}, several numerical examples indicate the efficiency of that method. Appendix~\ref{se:bloch} briefly introduces the Bloch transform and some of its mapping properties. 

\textit{Notation:} We denote quasiperiodic Sobolev spaces with regularity $s$ and quasiperiodicity $\alpha$ by  $H^s_\alpha$, such that $H^s_0$ denotes a \textit{periodic} Sobolev space and \emph{not} a space of functions that vanish on some boundary. Despite functions in Sobolev spaces are merely defined almost everywhere, we usually omit to write this down. Moreover, $C$ is a generic constant with value that might change from one appearance to the other. 

\section{Locally Perturbed Periodic Surface Scattering}\label{se:scatter}
In this section, we model scattering from a local perturbation $\Gamma_\pert$ of a periodic surface $\Gamma \subset \R^2$. Suppose $\Gamma:=\{{(y_1,\,\zeta(y_1))}:\,y_1\in\R\}$ is defined by a $\Lambda$-periodic Lipschitz continuous function $\zeta: \, \R \to \R$, i.e., $\zeta(x_1+\Lambda)=\zeta(x_1)$ for any $x_1\in\R$ and itself defines the periodic domain $\Omega :=\{ {(y_1,y_2)^\top}:\,y_1\in\R, \, y_2 > \zeta(y_1) \}$ above the graph of $\Gamma$. The locally perturbed periodic surface $\Gamma_\pert$ is then defined via a second Lipschitz continuous function  $\zeta_\pert : \, \R\to\R$ that satisfies 
\begin{equation}\label{eq:zetaPertDefEq} 
  \zeta_\pert (y_1) = \zeta(y_1) 
  \quad \text{for all } y_1 \not \in \left[-\frac{\Lambda}2, \frac{\Lambda}2 \right] \text{ and defines} \quad 
  \Gamma_\pert :=\{{(y_1,\,\zeta_\pert (y_1))}:\, y_1\in\R\}.
\end{equation}
We further assume without loss of generality that there is $H_0>0$ such that both Lipschitz surfaces $\Gamma_\pert$ and $\Gamma$ are included in $\R\times (0,H_0)$ and introduce the perturbed periodic domain $\Omega_\pert  := \{{(y_1,y_2)^\top}:\,y_1\in\R, \, y_2 > \zeta_\pert (y_1) \}$.
Then $\overline{\Omega}\subset\{y\in\R^2:\,y_2>0\}$ and $\overline{\Omega_\pert } \subset \{y\in\R^2:\,y_2>0\}$ as well.
For some positive number $H \geq H_0$ we further introduce truncated domains 
\begin{equation}\label{eq:OmegaH}
\Omega_H=\{y\in\Omega:\,y_2<H\}
\quad \text{and} \quad
\Omega_H^\pert =\{y\in\Omega_\pert :\,y_2<H\}.
\end{equation}
Note that we have for simplicity assumed that $\zeta\neq\zeta_\pert$ only in $(-{\Lambda}/{2},{\Lambda}/{2})$! 

The scattering problem we consider is described by the Helmholtz equation with Dirichlet boundary condition for the total wave field $u: \, \Omega_\pert  \to \C$,
\begin{equation}
\label{eq:helmholtz-tot}
\Delta u+k^2 u=0 \text{ in }\Omega_\pert ,\qquad u=0\text{ on }\Gamma_\pert ,
\end{equation}
where $k>0$ is the wavenumber and $u^i$ is the incident field. Moreover, the scattered field $u^s:=u-u^i$ satisfies the so-called angular spectrum representation,
\begin{equation}
\label{eq:URC}
u^s(x)=\frac{1}{2\pi}\int_\R e^{\i x_1\xi+\i\sqrt{k^2-|\xi|^2}(x_2-H_0)}\dhat{u}^s(\xi,H_0)\, d\xi\quad\text{ for }x_2>H_0.
\end{equation}
Here, $\sqrt{\cdot}$ is the square root extended to the complex plane by a branch cut at the negative imaginary axis (such that its real part and imaginary part are non-negative for numbers in the upper complex half-plane), and $\dhat{u}^s(\xi,H_0)$ is the Fourier transform of $u^s|_{x_2=H_0}$, i.e.,
\begin{equation}
\label{eq:fourierTrafo}
\dhat{\phi}(\xi):=\F\phi(\xi)=\frac{1}{\sqrt{2\pi}}\int_\R e^{-\i \xi x_1}\phi(x_1)\, dx_1\quad \text{ for }\phi\in \C^\infty_0(\R,\C) \text{ and } \xi\in\R,
\end{equation}
and extended by density to functions in $L^2(\R)$. Thus, we can define the exterior Dirichlet-to-Neumann map $T^+$, 
\begin{equation}\label{eq:T}
\frac{\partial u^s}{\partial x_2}(x_1,H)=\frac{i}{\sqrt{2\pi}}\int_\R\sqrt{k^2-|\xi|^2}e^{\i x_1\xi}\dhat{u}^s(\xi,H)\,  d\xi=:T^+(u^s|_{\Gamma_H})(x_1).
\end{equation}
Recall from Appendix~\ref{se:bloch} the spaces $H^{\pm 1/2}(\Gamma_H)$ and $H^1_r(\Omega_H)$ with its subspace $\tilde{H}^1_r(\Omega_H) = \{ u \in H^1_r(\Omega_H): \, u|_{\Gamma} = 0 \}$ of functions that vanish on $\Gamma$. 


The operator $T^+$ is bounded from $H^{1/2}_r(\Gamma_H)$ to $H^{-1/2}_r(\Gamma_H)$ for all $|r|<1$, see \cite{Chand2010}, and the the variational formulation of \eqref{eq:helmholtz-tot}-\eqref{eq:T} is to find $u\in \tilde{H}^1_r(\Omega_H^\pert)$ (that is, more precisely, the restriction of the total wave field to $\Omega_H$, but we omit this fact from now on) such that
\begin{equation}
\label{eq:var-tot}
\int_{\Omega_H^\pert }\left[\nabla u\cdot \nabla\overline{v}-k^2u\overline{v}\right] dx-\int_{\Gamma_H}T^+(u|_{\Gamma_H})\overline{v} \, ds=\int_{\Gamma_H}\left[\frac{\partial u^i}{\partial x_2}-T^+(u^i|_{\Gamma_H})\right]\overline{v} \, ds
\end{equation}
for all $v\in \tilde{H}^1(\Omega_H^\pert)$ with compact support in $\overline{\Omega_H}$.
Due to~\cite{Chand2010} we know that this variational problem is uniquely solvable for all $k>0$ and all bounded anti-linear right-hand sides. 

\begin{theorem}\label{eq:chand2010}
  For $|r|<1$ and any incident field $u^i \in H^1_r(\Omega_H)$, the variational problem~\eqref{eq:var-tot} possesses a unique solution $u \in \tilde{H}^1_r(\Omega_H)$. 
\end{theorem}

\section{Quasiperiodic Surface Scattering}  \label{se:PerSca}
The (Floquet-)Bloch transform $\J_\Omega$ reduces differential equations involving periodicity to, roughly speaking, quasiperiodic scattering problems from the unit cell of the periodic structure. 
Before exploiting this reduction, we need to recall some results on non-perturbed periodic scattering. In this section, all proofs are omitted and we refer to Appendix~\ref{se:bloch} and the references therein.

For an incident solution $u^i$ to the Helmholtz equation $\Delta u^i + k^2 u^i = 0$ in $\Omega$, the Dirichlet scattering problem from the periodic surface $\Gamma$ defined in the last section is described  for the total wave field $u: \, \Omega \to \C$ as in~\eqref{eq:var-tot} by 
\begin{equation}
\label{eq:ScaPer}
\Delta u+k^2u=0\text{ in }\Omega,\qquad u=0\text{ on }\Gamma,
\end{equation}
subject to the radiation condition \eqref{eq:URC} for the restriction of $u^s := u - u^i$ to $\Gamma_H$ for some $H>H_0$. The variational formulation of this problem is hence to find a solution $u \in \tilde{H}^1_r (\Omega_H)$ to~\eqref{eq:var-tot} with $\Omega_H^\pert $ replaced by $\Omega_H$, that is, 
\begin{equation}\label{eq:varFormHEScal}
  \int_{\Omega_H} \left[ \nabla u \cdot \nabla \overline{v} - k^2 u \,\overline{v} \right] \d{x}
  - \int_{\Gamma_H}  T^+ [ u ]\big|_{\Gamma_H} \overline{v} \, ds
  = \int_{\Gamma_H} \left[ \frac{\partial u^i}{\partial x_2} - T^+ [u^i]\big|_{\Gamma_H} \right] \overline{v} \, ds
\end{equation}
for all $v \in \tilde{H}^1_r (\Omega_H)$ with compact support in $\overline{\Omega_H}$.

The Bloch transform of a solution $u$ to the surface scattering problem~\eqref{eq:varFormHEScal} involving the periodic surface $\Gamma$ solves a corresponding quasiperiodic scattering problem.
To introduce the corresponding variational formulation, we recall the Wigner-Seitz cell $\W = (-\Lambda/2,\Lambda/2]$ of periodicity $\Lambda$ and the periodic sets 
\[
  \Omega_H^\Lambda = \{ x \in \Omega_H: \, x_1 \in \W \}, 
  \quad 
  \Gamma_\Lambda = \{ x \in \Gamma: \, x_1 \in \W \},
  \text{ and} \quad 
  \Gamma_H^\Lambda = \{ x \in \Gamma_H: \, x_1 \in \W \},
\]
as well as Sobolev spaces $\tilde{H}^1_\alpha(\Omega_H^\Lambda)$, $H^{s}_\alpha(\Gamma_H^\Lambda)$, and $H^r_\0(\Wast; \tilde{H}^1_\alpha(\Omega_H^\Lambda))$ for quasiperiodicity $\alpha \in \Wast = (-\Lambda^\ast,\Lambda^\ast] := (-\pi/\Lambda, \pi/\Lambda]$ from Appendix~\ref{se:bloch}; $\Wast = (-\pi/\Lambda, \pi/\Lambda]$ is the so-called Brillouin zone. 
We first rely on a well-known periodic Dirichlet-to-Neumann operator $T_\alpha^+$ on $\Gamma_H^\Lambda\subset \Gamma_H$ that is continuous from $H^{s+1}_\alpha(\Gamma_H^\Lambda)$ into $H^{s}_\alpha(\Gamma_H^\Lambda)$ for all $s\in \R$,
\begin{equation}
   T_\alpha^+ \left( \varphi \right)\Big|_{\Gamma_H^\Lambda}
   = \bigg[\i \sum_{j\in\Z} \sqrt{k^2 - |\Lambda^\ast j - \alpha|^2} \, \hat{\varphi}(j) \, e^{\i (\Lambda^\ast j-\alpha)  x_1} \bigg]\bigg|_{\Gamma_H^\Lambda}
   \quad \text{for }
   \varphi = \sum_{j\in\Z} \hat{\varphi}(j) e^{\i (\Lambda^\ast j - \alpha)  x_1}.
\end{equation}
Obviously, $T_\alpha^+$ is a periodic version of the operator $T^+$ from~\eqref{eq:T}.
Second, we introduce a bounded sesqui-linear form $a_\alpha$ on $\tilde{H}^1_\alpha(\Omega_H^\Lambda) \times \tilde{H}^1_\alpha(\Omega_H^\Lambda)$ corresponding to the Helmholtz equation with Dirichlet boundary condition on $\Gamma$,
\[
  a_\alpha(w,v)
  := \int_{\Omega_H^\Lambda} \Big[ \nabla_x w \cdot \nabla_x \overline{v} - k^2 w(\alpha, \cdot) \, \overline{v} \Big] \d{x}
  - \int_{\Gamma_H^\Lambda} T_\alpha^+ \left( w|_{\Gamma_H^\Lambda} \right) \, \overline{v}|_{\Gamma_H^\Lambda} \\, ds,
\]
and state an equivalence result that can be shown along the lines of Theorem~9 in~\cite{Lechl2016}.

\begin{theorem}\label{th:equiScalarPerio}
Suppose $u^i$ belongs to $H^1_r(\Omega_H)$ for some $r \in [0,1)$. 
Then a function $u \in \tilde{H}^1_r(\Omega_H^\Lambda)$ solves~\eqref{eq:varFormHEScal} if and only if $w := \J_\Omega u \in H^r_\0(\Wast; \tilde{H}^1_\alpha(\Omega_H^\Lambda))$ solves
\begin{equation}\label{eq:heAlpha}
  a_\alpha(w(\alpha,\cdot),v)
  = \int_{\Gamma_H^\Lambda} f_\alpha \overline{v} \\, ds \quad\text{ for }  f_\alpha= \frac{\partial }{\partial \nu} \J_\Omega u^i (\alpha,\cdot)- T_\alpha^+ \left[\J_\Omega u^i (\alpha,\cdot)\big|_{\Gamma_H^\Lambda} \right]
\end{equation}
for all $v \in \tilde{H}^1_\alpha(\Omega_H^\Lambda)$ and almost every $\alpha \in \Wast$.
\end{theorem}

The last theorem's assumption that $u^i$ belongs to $H^1_r(\Omega_H)$ for some $r \geq 0$ is in two dimensions satisfied, e.g., for (non-periodic) point sources or Herglotz wave functions with, roughly speaking, vanishing horizontal part, see~\cite{Lechl2016a}. The periodic scattering problem is always uniquely solvable if, e.g., $\Gamma$ is graph of a Lipschitz function, see~\cite{Bonne1994, Elsch2002}. 

\begin{lemma}\label{th:uniqueAlpha}
If $\Gamma$ is graph of a Lipschitz continuous function, then~\eqref{eq:heAlpha} is solvable for all $(k_\ast,\alpha) \in (0,\infty) \times \Wast$ and the solution operators $A_\alpha$ are uniformly bounded in $\alpha \in \Wast$.
\end{lemma}

The solution $w=w(\alpha,\cdot)$ does generally not belong to $H^1_\0(\Wast; \tilde{H}^1(\Omega_H^\Lambda))$, which follows actually already from~\cite{Chand2010}.

\begin{theorem} \label{th:exiSolScalPerio}
Assume that $\Gamma$ is graph of a Lipschitz continuous function. 
If $u^i \in H^1_r(\Omega_H)$ for $r\in [0,1)$, then the solution $w=w(\alpha,\cdot)$ to~\eqref{eq:heAlpha} belongs to $H^r_\0(\Wast; \tilde{H}^1_\alpha(\Omega_H^\Lambda))$ and the solution $u = \J_\Omega^{-1} w$ to~\eqref{eq:varFormHEScal} belongs to $\tilde{H}^{1}_r(\Omega_H)$. If $r>1/2$, then $\alpha \mapsto w(\alpha,\cdot)$ is continuous from $\Wast$ into $\tilde{H}^1_\alpha(\Omega_H^\Lambda)$. 
\end{theorem}
\begin{proof}
We merely show the continuity result: 
Reference~\cite{Chand2010} states that the solution $u$ to~\eqref{eq:ScaPer} belongs to $\tilde{H}^1_r(\Omega_H)$ if the incident field decays as indicated for $r \in [0,1)$ (even for $r \in (-1,1)$). 
The transformed solution $w = \J_\Omega u$ hence belongs to $H^r_\0(\Wast; \tilde{H}^1_\alpha(\Omega_H^\Lambda))$.
If $r>1/2$, such functions are continuous in $\alpha$ due to Sobolev's embedding theorem (or Morrey's estimate) in one dimension (see~\cite{Evans1998, Lechl2016a}), such that $\alpha \mapsto w(\alpha,\cdot)$ is continuous from $\Wast$ into $\tilde{H}^1_\alpha(\Omega_H^\Lambda) \subset \tilde{H}^1(\Omega_H^\Lambda)$. 
(The norm in $\tilde{H}^1_\alpha(\Omega_H^\Lambda)$ is simply the norm of $\tilde{H}^1(\Omega_H^\Lambda)$!) 
In particular, the evaluations $w(\alpha,\cdot)$ in $\tilde{H}^1_\alpha(\Omega_H^\Lambda)$ depend continuously on $\alpha$. 
\end{proof}

\section{Periodized Quasiperiodic Scattering Problems}\label{se:qpScattPerturb}
Now we start to analyze scattering problems from locally perturbed periodic surfaces based on the Bloch transform from Appendix~\ref{se:bloch} and our knowledge on quasiperiodic scattering from Section \ref{se:PerSca}. As the variational formulation \eqref{eq:var-tot} of the locally perturbed periodic surface scattering problem is set in the non-periodic space $\tilde{H}^1(\Omega_H^\pert)$, we have to transform it into a problem formulated in the periodic domain $\Omega_H$. 

To this end, we use the diffeomorphism $\Phi_\pert$ from $\Omega_H$ into $\Omega_H^\pert$, defined by
\begin{equation}\label{eq:phiPertDef}
  \Phi_\pert:\, x\mapsto \left(x_1,x_2+\frac{(x_2-H)^3}{(\zeta(x_1)-H)^3}(\zeta_p(x_1)-\zeta(x_1)\right).
\end{equation}
%
%
The support of $\Phi_\pert-I$ is contained in $\Omega_H^\Lambda$ as the support of $\zeta_\pert - \zeta$ is by assumption included in $[- \Lambda/2, \, \Lambda/2]$, too, see~\eqref{eq:zetaPertDefEq}. The transformed total field $u_\T = u\circ\Phi_\pert\in {\tilde{H}^1_r(\Omega_H)}$ then satisfies by the transformation theorem the following variational problem in $\Omega_H$, 
\begin{equation}
\label{eq:var-tra}
\int_{\Omega_H} \hspace*{-1mm} \left[A_\pert \nabla u_\T \cdot \nabla\overline{v_\T} - k^2c_\pert \, u_\T\overline{v_\T} \right] dx 
- \int_{\Gamma_H} \hspace*{-1mm} T^+(u_\T|_{\Gamma_H})\overline{v_\T} \, ds 
= \int_{\Gamma_H} \hspace*{-1mm} \left[\frac{\partial u^i}{\partial x_2}-T^+(u^i|_{\Gamma_H})\right]\overline{v_\T} \, ds
\end{equation}
for all $v_\T \in \tilde{H}^1(\Omega_H)$ with compact support in $\overline{\Omega_H}$ and coefficients 
\begin{align*}
A_\pert (x)&:=\big|\det\nabla\Phi_\pert(x)\big|\big[(\nabla\Phi_\pert(x))^{-1}((\nabla\Phi_\pert(x))^{-1})^T\big] \in L^\infty(\Omega_H,\R^{2\times  2}),\\
c_\pert (x)&:=\big|\det\nabla\Phi_\pert(x)\big|\in L^\infty(\Omega_H).
\end{align*}
We reformulate \eqref{eq:var-tra} by applying the inverse Bloch transform composed with the Bloch transform to the weak solution $u_\T$. As $\nabla \Phi_\pert = I$ outside $\Omega_H^\pert $ there holds that $A_\pert -I$ and $c_\pert -1$ are both supported in $\Omega_H^\Lambda$, and an explicit computation shows that the Bloch transform of $(A_\pert -I)\nabla u_\T$ equals to $(\Lambda/2\pi)^{1/2}(A_\pert -I)\nabla u_\T$ in the space $L^2(\Omega_H)^2$, and the Bloch transform of $(c_\pert -1)u_\T$ is $(\Lambda/2\pi)^{1/2}(c_\pert -1) u_\T$ in $L^2(\Omega_H)$. (Despite, both functions have compact support in $\Omega_H^\Lambda$.) 

If we assume that the incident field $u^i$ belongs to $H^1_r(\Omega_H)$ for some $r \in [0,1)$, then the Bloch transform $w_\B = \J_\Omega u_\T$ belongs to $L^2(\Wast;\tilde{H}^1_{\alpha}(\Omega_H^\Lambda))$ and satisfies for all test functions $v_\B \in L^2(\Wast;\tilde{H}^1_{\alpha}(\Omega_H^\Lambda))$ that 
\begin{equation} \label{eq:var-traBloch}
\begin{split}
\int_\Wast & a_{\alpha} (w_\B(\alpha,\cdot),v_\B(\alpha,\cdot)) \, d\alpha 
+ \left[\frac{\Lambda}{2\pi}\right]^{1/2} \int_{\Omega_H^\Lambda} (A_\pert -I)\nabla \big(\J_\Omega^{-1} w_\B \big) \cdot \nabla \big(\overline{\J_\Omega^{-1} v_\B} \big)\, dx \\
& - k^2 \left[\frac{\Lambda}{2\pi}\right]^{1/2}\int_{\Omega_H^\Lambda} (c_\pert -1) \J_\Omega^{-1} w_\B  \, \overline{\J_\Omega^{-1} v_\B} \, dx 
  = \int_\Wast \int_{\Gamma_H^\Lambda} f(\alpha,\cdot) \, \overline{v_\B}(\alpha,\cdot) \, ds \, d\alpha, 
\end{split}
\end{equation}
for the right-hand side $f \in H^r_0(\Wast;H^{-1/2}_\alpha(\Gamma_H^\Lambda))$ with $f(\alpha,\cdot)  \in H^{-1/2}_\alpha(\Gamma_H^\Lambda)$ given by 
\begin{equation}\label{eq:heAlphaf}
  f(\alpha,\cdot) = \frac{\partial \J_\Omega u^i(\alpha,\cdot)}{\partial x_2} - T^+_\alpha\left[(\J_\Omega u^i)(\alpha,\cdot)|_{\Gamma_H^\Lambda}\right].
\end{equation} 
The corresponding coupled strong formulation is 
\begin{equation*} 
\Delta_x w_\B (\alpha,\cdot)+k^2w_\B (\alpha,\cdot)= - \left[\frac{\Lambda}{2\pi}\right]^{1/2}\Big[\nabla\cdot\left[(A_\pert -I)\nabla \big(\J_\Omega^{-1} w_\B\big) \right]+k^2(c_\pert -1) \big(\J_\Omega^{-1} w_\B\big) \Big]
\end{equation*}
with boundary conditions $w_\B(\alpha,\cdot)=0$ on $\Wast\times\Gamma^\Lambda$ and $\partial w_\B(\alpha,\cdot)/\partial x_2 - T^+_{\alpha}\big[w_\B(\alpha,\cdot)|_{\Gamma^\Lambda_H}\big] = \partial (\J_\Omega u^i (\alpha,\cdot))/\partial x_2 - T^+_{\alpha}\big[(\J_\Omega u^i)(\alpha,\cdot)|_{\Gamma^\Lambda_H}\big]$ on $\Wast\times\Gamma^\Lambda_H$. 

\begin{theorem}\label{th:equivalent}
Assume that the incident field $u^i$ belongs to $H^1_r(\Omega_H)$ for some $r\in [0,1)$. Then $u_\T \in\tilde{H}^1_r(\Omega_H)$ satisfies \eqref{eq:var-tra} if and only if $w_\B = \J_\Omega u_\T \in H^r_\0(\Wast;\tilde{H}^1_\alpha(\Omega_H^\Lambda))$ satisfies \eqref{eq:var-traBloch}. 
\end{theorem}
\begin{proof}
From the arguments before \eqref{eq:var-traBloch}, it is easy to see that if $u_\T$ satisfies \eqref{eq:var-tra}, then $w_\B$ solves \eqref{eq:var-traBloch}.
If $w_\B \in H^r_\0(\Wast;\tilde{H}^1_\alpha(\Omega_H^\Lambda))$ satisfies \eqref{eq:var-traBloch}, then the property $\J_\Omega^{-1}=\J_\Omega^*$ implies that $u_\T = \J_\Omega^{-1} u_\B \in\tilde{H}^1_r(\Omega_H)$ satisfies \eqref{eq:var-tra}.
\end{proof}

We next consider unique solvability of~\eqref{eq:var-traBloch}.

\begin{theorem}\label{th:RegQuasi}
If $\Gamma_\pert$ is graph of a Lipschitz continuous function, then~\eqref{eq:var-traBloch} is uniquely solvable in $H^r_\0(\Wast; \tilde{H}^1_\alpha(\Omega_H^\Lambda))$ for all incident fields $u^i \in H^1_r(\Omega_H)$ for $r \in [0,1)$.
\end{theorem}
\begin{proof}
If $\Gamma_\pert$ is graph of a Lipschitz continuous function, then~\cite{Chand2010} implies that both variational formulations~\eqref{eq:var-tot} and, equivalently,~\eqref{eq:var-tra} are uniquely solvable in $\tilde{H}^1_r(\Omega_H)$ for incident fields in $u^i \in H^1_r(\Omega_H)$, $r \in (-1,1)$. 
Theorem~\ref{th:equivalent} now implies that~\eqref{eq:var-traBloch} is uniquely solvable, too.  
\end{proof}

Before we study error estimates for a discretization of the variational formulation of $w_\B$ in the next section, we need to show an auxiliary result on the regularity of this solution. 

\begin{theorem}\label{th:regulH2}
Assume that the restrictions of $\partial u^i/\partial x_2$ and $u^i$ to $\Gamma_H$ belong to $H^{1/2}_r(\Gamma_H)$ and to $H^{3/2}_r(\Gamma_H)$ for $r \in [0,1)$, and that $\zeta$ and $\zeta_\pert  \in {C^{2,1}(\R,\R)}$. Then the quasiperiodic solutions $w_\B(\alpha,\cdot)$ to~\eqref{eq:var-traBloch} do all belong to $H^2_\alpha(\Omega_H^\Lambda)$ and the inverse Bloch transformation $u_\T = \J_\Omega^{-1} w_\B$ belongs to $H^2(\Omega_H)$.
\end{theorem} 
\begin{proof}
From~\cite{Chand2005} we know that the variational problem~\eqref{eq:var-tot} possesses a unique solution that is bounded in $H^1(\Omega_H)$ by the norm of $u^i$ in $H^{1/2}(\Gamma_H)$. 
From the regularity of $\zeta$ and $\zeta_\pert$, we deduce that $\Gamma_\pert$ is $C^{2,1}$-smooth such that elliptic regularity results, see, e.g.,~\cite{McLea2000}, imply that both $\partial u^i / \partial x_2$ and the restriction of $u^i$ to $\Gamma_H$ itself belong actually to $H^{3/2}(\Gamma_H)$. 
In turn, these regularity results further imply by localization that $u \in H^2(\Omega_H)$ (see, e.g.,~\cite{Lechl2009}). 
Thus, $u_\T = u \circ \Phi_\pert$ belongs to $H^2(\Omega_H)$ by the $C^{2,1}$-smoothness of $\Phi_\pert$ defined via $\zeta_\pert$ and $\zeta$ in~\eqref{eq:phiPertDef}, and its Bloch transform belongs to $L^2(\Wast; H^2_\alpha(\Omega_H^\Lambda))$ by the mapping properties of $\J_\Omega$:  $w_\B = \J_\Omega u_\T \in L^2(\Wast; H^2_\alpha(\Omega_H^\Lambda))$. 
For almost every $\alpha$, the solution $w_\B$ to~\eqref{eq:var-traBloch} hence belongs to $H^2_\alpha(\Omega_H^\Lambda)$
\end{proof}

We finally state an equivalent way of writing~\eqref{eq:var-traBloch} if the incident field $u^i$ belongs to $H^1_r(\Omega_H)$ for some $r \in (1/2,1)$. 

\begin{theorem}\label{th:testFuncs}
If $\Gamma_\pert$ is graph of a Lipschitz continuous function and if $u^i \in H^1_r(\Omega_H)$ for $r \in (1/2,1)$, then the  solution $w_\B \in L^2(\Wast; \tilde{H}^1_\alpha(\Omega_H^\Lambda))$ equivalently satisfies for all $\alpha\in\Wast$ and all $v_\alpha \in \tilde{H}^1_\alpha(\Omega_H^\Lambda)$ that 
\begin{align}\label{eq:altForm}
  a_{\alpha} (w_\B(\alpha, \cdot), v_\alpha) 
  & + \left[\frac{\Lambda}{2\pi}\right]^{1/2} \int_{\Omega_H^\Lambda} (A_\pert -I)\nabla \big(\J_\Omega^{-1} w_\B\big) \cdot \nabla \overline{v_\alpha} \, dx \\
  & - \left[\frac{\Lambda}{2\pi}\right]^{1/2} k^2 \int_{\Omega_H^\Lambda} (c_\pert -1) \J_\Omega^{-1} w_\B \, \overline{v_\alpha} \, dx 
  = \int_{\Gamma_H^\Lambda} f(\alpha,\cdot) \, \overline{v_\alpha} \, ds. \nonumber
\end{align} 
\end{theorem}
%
%
\begin{proof}
Reference~\cite{Chand2010} states that the solution $u$ to~\eqref{eq:var-tot} belongs to $H^1_r(\Omega_H)$ if the incident field decays as indicated for $r \in [0,1)$ (even for $r\in (-1,1)$). 
As $\Phi_\pert$ merely modifies $u$ in a bounded region, the transformed field $u_\T = u \circ \Phi_\pert$ decays with the same rate as $x_1 \to \pm \infty$ and $w_\B = \J_\Omega u_\T$ hence belongs to $H^r_\0(\Wast; \tilde{H}^1_\alpha(\Omega_H^\Lambda)) \subset H^r_\0(\Wast; \tilde{H}^1(\Omega_H^\Lambda))$ for $r >1/2$.
Such functions are continuous in $\alpha$ with values in $\tilde{H}^1(\Omega_H^\Lambda)$ due to Sobolev embeddings (or Morrey's estimate) in one dimension (see~\cite{Evans1998, Lechl2016a}). 

Thus, we can test equation~\eqref{eq:var-traBloch} by Dirac distributions in $\alpha_0 \in \Wast$, multiplied by functions $v_{\alpha_0} \in \tilde{H}^1_\alpha(\Omega_H^\Lambda)$, to get that $w_\B$ solves~\eqref{eq:altForm} for each $\alpha_0 \in\Wast$. 
In turn, if $w_\B$ satisfies the latter (infinite number of) equations, constructing a complete countable family of test functions in $L^2(\Wast; \tilde{H}^1_\alpha(\Omega_H^\Lambda))$ shows that $w_\B$ solves~\eqref{eq:var-traBloch} as well. (To this end, one takes an orthonormal family of the separable Hilbert space $L^2(\Wast; \tilde{H}^1_\alpha(\Omega_H^\Lambda))$; separability can be shown by first considering functions that are piecewise constant in $\alpha$ and take values in the periodic Sobolev functions $\tilde{H}^1_0(\Omega_H)$, and, second, multiplication of these functions by $\exp(- \i \alpha x_1)$.) 
\end{proof}

\section{The Numerical Scheme and Error Estimates}\label{se:invBloch}
In this section, we discuss a Galerkin discretization of the variational formulation~\eqref{eq:var-traBloch} of $w_\B$ together with an error estimate for the solution to the discretized problem. 
Of course, this makes it necessary to introduce a suitable finite element space first. 
We actually chose the simplest type of (nodal) elements, which is not crucial but avoids technicalities. (For instance,  when using periodic boundary integral equations instead, we would need to take care of exceptional wave numbers where uniqueness of solution fails.)

We assume hence to know a family of regular and quasi-uniform meshes $\mathcal{M}_h$, $0<h\leq h_0$, of the domain $\Omega_H^\Lambda$ such that for each mesh width $h$ the nodes on the right and left boundary of $\Omega_H^\Lambda$ have the same height. This in particular ensures that piecewise linear and globally continuous functions on that mesh can be extended to periodic functions on a regular and quasi-uniform mesh of $\Omega_H$. To construct such periodic functions we omit now all nodal points on the left boundary of $\Omega_H^\Lambda$, denote the piecewise linear and globally continuous nodal functions equal to one at exactly one of the remaining nodes and zero at all others by $\{ \phi_{M}^{(\ell)} \}_{\ell=1}^{M}$, and denote the discrete subspace spanned by these functions by $\tilde{V}_h \subset \tilde{H}^1_0(\Omega_H^\Lambda)$. 
It is well-known (see, e.g.,~\cite{Saute2007}) that for functions $v \in \tilde{H}^1_0(\Omega_H^\Lambda) \cap H^2(\Omega_H^\Lambda)$ there holds 
\begin{equation}\label{eq:assSpac}
  \min_{v_h \in \tilde{V}_h} \| v_{\alpha,h} - v \|_{H^\ell(\Omega_H^\Lambda)} \leq C h^{2-\ell} \| v \|_{H^2(\Omega_H^\Lambda)}
  \qquad \text{for $0<h<h_0$.} 
\end{equation}
To introduce our finite element space, we introduce uniformly distributed grid points 
\[
  \alpha_N^{(1)} =-\frac{\pi}{\Lambda}+\frac{\pi}{N\Lambda}, \qquad 
  \alpha_N^{(j)} = \alpha_N^{(j-1)}+\frac{2\pi}{N\Lambda}  \in \Wast \quad
  \text{for }j=2,\dots,N \in \N,
\] 
and consider a basis $\{ \psi_{N}^{(j)} \}_{j=1}^N$ of the space of functions that are piecewise constant on each interval $[\alpha_N^{(j)}-\pi/(N\Lambda), \alpha_N^{(j)}+\pi/(N\Lambda)]$ for $j=1,\dots,N$ such that $\psi_N^{(j)}$ equals one on the $j$th interval and zero else. 
The finite element space $\tilde{X}_{N,h}$ we consider is spanned by products of these two bases, multiplied by $\exp(-\i \alpha x_1)$, 
\begin{equation}\label{eq:XN}
  \tilde{X}_{N,h} = \bigg\{ v_{N,h}(\alpha,x) = e^{-\i \alpha x_1} \sum_{j=1}^N \sum_{\ell=1}^M \hspace*{-0.7mm} v_{N,h}^{(j,\ell)} \psi_{N}^{(j)}(\alpha) \phi_{M}^{(\ell)}(x)\hspace*{-0.7mm} : v_{N,h}^{(j,\ell)} \in \C \bigg\}
  \subset L^2(\Wast; \tilde{H}^1_\alpha(\Omega_H^\Lambda)),
\end{equation}
It is easy to see that functions in $\tilde{X}_{N,h}$ indeed belong to $L^2(\Wast; \tilde{H}^1_\alpha(\Omega_H^\Lambda))$: 
Without the exponential factor in~\eqref{eq:XN} they are clearly periodic in $\alpha$ and in $x_1$ as functions defined in $\R \times \Omega_H$; further, multiplication by $\exp(-\i \alpha x_1)$ implies for $\alpha \in \R$ and $x \in \Omega_H$ that 
\[
  v_{N,h}\left(\alpha, \left( \begin{smallmatrix}x_1+\Lambda\\ x_2 \end{smallmatrix} \right) \right) 
  = e^{-\i \alpha (x_1+\Lambda)} \sum_{j=1}^N \sum_{\ell=1}^M v_{N,h}^{(j,\ell)} \psi_{N}^{(j)}(\alpha) \phi_{M}^{(\ell)}(\left( \begin{smallmatrix}x_1+\Lambda\\ x_2 \end{smallmatrix} \right))
  = e^{-\i \alpha \Lambda} v_{N,h}\left(\alpha, x \right),
\]
such that $v_{N,h}(\alpha,\cdot)$ is $\alpha$-quasiperiodic and hence belongs to $\tilde{H}^1_\alpha(\Omega_H^\Lambda)$. 

Introduce now, abstractly, the sesqui-linear form
\[
  b(w,v) =\left[\frac{\Lambda}{2\pi}\right]^{1/2}\int_{\Omega_H^\Lambda}\left[(A_\pert -I)\nabla { w}\cdot\overline{\nabla v}-k^2(c_\pert -1) \, w \overline{v}\right] dx
  \quad \text{on } H^1(\Omega_H^\Lambda) \times H^1(\Omega_H^\Lambda). 
\]
For the boundary term $f(\alpha, \cdot) = \partial (\J_\Omega u^i(\alpha,\cdot))/\partial x_2 - T_\alpha^+(\J_\Omega u^i)(\alpha,\cdot)$ in $H^{-1/2}_\alpha(\Gamma_H)$ from~\eqref{eq:heAlphaf}, we now seek a finite element solution $w_{N,h} \in \tilde{X}_{N,h}$ to the finite-dimensional problem  
\begin{equation}\label{eq:var-traBloch3}
\int_\Wast a_{\alpha}(w_{N,h},v_{N,h}) \, d\alpha 
+ b \big( \J_{\Omega}^{-1} w_{N,h}, \overline{\J_{\Omega}^{-1} v_{N,h}} \big) 
 = \int_\Wast \int_{\Gamma_H^\Lambda} f(\alpha, \cdot) \, \overline{v}_{N,h} \, ds \, d\alpha 
\end{equation}
for all $v_{N,h} \in \tilde{X}_{N,h}$.
As functions in $\tilde{X}_{N,h}$ are for fixed $x$ piecewise exponential in $\alpha$ on each interval $[\alpha_N^{(j-1)}, \alpha_N^{(j)}]$, the inverse Bloch transform in the latter problem can be explicitly computed:  
\begin{align}
  \J_\Omega^{-1} w_{N,h} (\alpha,x)
  & = \left[ \frac{\Lambda}{2\pi} \right]^{1/2} \sum_{j=1}^N \int_{\alpha_N^{(j)}-\pi/(N\Lambda)}^{\alpha_N^{(j)}+\pi/(N\Lambda)} w_{N,h}(\alpha, x) \d{\alpha} \nonumber \\
  & = \left[ \frac{\Lambda}{2\pi} \right]^{1/2} \sum_{j=1}^N \sum_{\ell=1}^M v_{N,h}^{(j,\ell)} \phi_{M}^{(\ell)}(x) \int_{\alpha_N^{(j)}-\pi/(N\Lambda)}^{\alpha_N^{(j)}+\pi/(N\Lambda)} e^{-\i \alpha x_1} \d{\alpha} \nonumber \\
  & = \left[ \frac{\Lambda}{2\pi} \right]^{1/2} \sum_{j=1}^N g_N^{(j)} (x_1) \sum_{\ell=1}^M v_{N,h}^{(j,\ell)} \phi_{M}^{(\ell)}(x) 
  =: \J_{\Omega,N}^{-1} \big( \{ w_{N,h} (\alpha_N^{(j)},\cdot)\}_{j=1}^N \big) \label{eq:aux462}
\end{align}
where 
\begin{equation} \label{eq:aux481}
  g_N^{(j)}(x_1) 
  =  \i e^{-\i \alpha_N^{(j)} x_1} \left[e^{-\i \pi x_1 /{(N\Lambda})} - e^{\i \pi x_1 /{(N\Lambda})} \right] /  x_1 \quad \text{ if $x_1 \not = 0$,} 
\end{equation}
and $g_N^{(j)}(0) = 2\pi/(N\Lambda)$. 
Note that~\eqref{eq:aux462} hence defines a numerical approximation $\J_{\Omega,N}^{-1}$ to the inverse Bloch transform that we rely on in our numerical examples later on.
%
%

\begin{theorem}\label{th:numInvBloch}
Assume that $u^i \in H^2_r(\Omega_H)$ for $r \geq 1/2$ and that $\zeta$ and $\zeta_\pert$ are $C^{2,1}$ diffeomorphisms. 
Then the linear system \eqref{eq:var-traBloch3} is uniquely solvable in $\tilde{X}_{N,h}$ for any right-hand side 
\[
  f(\alpha, \cdot) = \frac{\partial \J_\Omega u^i}{\partial x_2}(\alpha,\cdot) - T_\alpha^+(\J_\Omega u^i)(\alpha,\cdot)
  \quad \text{in } H^r_\0(\Wast; H^{1/2}_\alpha(\Gamma_H^\Lambda))
\] 
if $N \geq N_0$ is large enough and $0<h<h_0$ is small enough.
The solution $w_\B \in \tilde{X}_{N,h}$ satisfies the error estimate 
\begin{equation}\label{eq:errorU}
\big\| w_{N,h} - w_\B \big\|_{L^2(\Wast; H^\ell(\Omega_H^\Lambda))} 
\leq C h^{1-\ell} \left( N^{-r} + h \right) \| f \|_{H^r_\0(\Wast; H^{1/2}_\alpha(\Gamma_H^\Lambda))}, 
\qquad \ell=0,1.
\end{equation}
\end{theorem}

\begin{remark}
  (a) Despite we have explicitly introduced the finite dimensional space $\tilde{X}_{N,h}$ via piecewise linear and globally continuous functions on a mesh of $\Omega_H^\Lambda$, Theorem~\ref{eq:errorU} holds for any family of finite-dimensional spaces that satisfies~\eqref{eq:assSpac}.
  Of course,~\eqref{eq:var-traBloch3} is also uniquely solvable for any other continuous linear form on $H^r_\0(\Wast; \tilde{H}^1_\alpha(\Omega_H^\Lambda))$ as right-hand side.\\[1mm] 
  (b) The assumption of Theorem~\ref{th:numInvBloch} for $u^i$ applies for instance if $u^i$ is the Dirichlet Green's function of the half space for all $r\in(1/2,1)$, see~\cite{Lechl2016a}.  
\end{remark}

\begin{proof}
The proof exploits the regularity result in Theorem~\ref{th:regulH2} stating that $w_\B(\alpha; \cdot) \in H^2_\alpha(\Omega_H^\Lambda)$. The latter function is continuous in $\alpha$ by Theorem~\ref{th:testFuncs}. 
Solvability of the given discretized sesqui-linear problem that features a continuous sesqui-linear form that satisfies a G\r{a}rding's inequality as well as an injectivity condition is due to basic finite elements theory, see, e.g.,~\cite{Saute2007}. 
The indicated error bound~\eqref{eq:errorU} follows from the corresponding standard convergence estimate
\begin{align}
  \big\| w_{N,h} - w_\B\big\|_{L^2(\Wast;\tilde{H}^1_\alpha(\Omega_H^\Lambda))}
  & \leq C \inf_{v_{N,h}\in \tilde{X}_{N,h}} \big\| v_{N,h} - w_\B\big\|_{L^2(\Wast;\tilde{H}^1_\alpha(\Omega_H^\Lambda))} \nonumber \\
  & \leq C (N^{-r} + h) \| w_\B \|_{H^r_\0(\Wast;H^2(\Omega_H^\Lambda))} \label{eq:aux494} \\
  & \leq C (N^{-r} + h) \| f \|_{H^r_\0(\Wast;H^{1/2}_\alpha(\Gamma_H^\Lambda))}. \nonumber 
\end{align}

To prove the additional $L^2(\Omega_H^\Lambda)$-estimate, we need to consider the adjoint problem to find $v_\B \in  L^2(\Wast; \tilde{H}^1_\alpha(\Omega_H^\Lambda))$ such that, for some given $g \in L^2(\Wast; L^2(\Omega_H^\Lambda))$, there holds   
\begin{equation}\label{eq:adjoint}
\begin{split} 
\int_\Wast a_{\alpha}(w,v_\B) \, d\alpha 
+ b \big( \J_\Omega^{-1} w, & \J_\Omega^{-1} v_\B \big) 
 = \int_\Wast \int_{\Omega_H^\Lambda} w \, \overline{g} \, ds \, d\alpha 
\end{split}
\end{equation}
for all $w \in L^2(\Wast; \tilde{H}^1_\alpha(\Omega_H^\Lambda))$. 
Conjugating the entire latter equation obviously yields a Fredholm problem such that is suffices to show uniqueness of solution to deduce existence of solution. 
{If $v \in L^2(\Wast; \tilde{H}^1_\alpha(\Omega_H^\Lambda))$ annihilates the latter sesqui-linear form, 
\begin{align*}
  \int_\Wast a_{\alpha}(v,v) \, d\alpha 
  + b \big( \J_{\Omega}^{-1} v, \, \overline{\J_{\Omega}^{-1}v} \big) =0,
\end{align*}
we  deduce that $v$ solves as well the homogeneous primal problem~\eqref{eq:var-traBloch} and hence vanishes by uniqueness of the primal problem (see Theorem~\ref{th:RegQuasi}). }

The arguments proving the $H^2$-regularity estimate for the solution $w_\B$ from Theorem~\ref{th:regulH2} directly transfer to the solution $v_\B$ to the adjoint problem~\eqref{eq:adjoint}, such that there is $C>0$ with $\| v_\B \|_{L^2(\Wast; H^2(\Omega_H^\Lambda)))} \leq C \| g \|_{L^2(\Wast; L^2(\Omega_H^\Lambda))}$. 
Recall that the difference $w_{N,h}-w_\B$ of the solutions to the continuous and discretized primal problem satisfy Galerkin orthogonality, 
\[
  \mathcal{A}(w_{N,h}-w_\B, v_{N,h})
  := \int_\Wast a_\alpha(w_{N,h}-w_\B, v_{N,h}) \, d\alpha 
  \, + \, b( \J_\Omega^{-1}(w_{N,h}-w_\B), \, \J_\Omega^{-1} v_{N,h})
  =0
\]
for all elements $v_{N,h} \in \tilde{X}_{N,h}$ of the discretization space. 
This shows that 
\begin{align}
 (w_{N,h}-w_\B, g)_{L^2(\Wast \times \Omega_H^\Lambda)}
 & = \mathcal{A}(w_{N,h}-w_\B, v_\B)
 = \mathcal{A}(w_{N,h}-w_\B, v_\B- v_{N,h}) \label{eq:aux543} \\
 & \leq C \| w_{N,h}-w_\B \|_{L^2(\Wast; \tilde{H}^1_\alpha(\Omega_H^\Lambda))} \, \| v_\B- v_{N,h} \|_{L^2(\Wast; \tilde{H}^1_\alpha(\Omega_H^\Lambda))} \nonumber 
\end{align}
for all $v_{N,h} \in \tilde{X}_{N,h}$. 
If we choose $v_{N,h}$ as the orthogonal projection of $v_\B$ onto $\tilde{X}_{N,h} \subset L^2(\Wast; \tilde{H}^1_\alpha(\Omega_H^\Lambda))$, then firstly 
\begin{equation}\label{eq:aux548}
  \| v_\B- v_{N,h} \|_{L^2(\Wast; \tilde{H}^1_\alpha(\Omega_H^\Lambda))}
  \leq C h \| v_\B \|_{L^2(\Wast; H^2(\Omega_H^\Lambda))} 
  \leq C h \| g \|_{L^2(\Wast; L^2(\Omega_H^\Lambda))}.  
\end{equation}
Together with~\eqref{eq:aux543}, this estimate secondly implies that  
\[
  (w_{N,h}-w_\B,g)_{L^2(\Wast \times \Omega_H^\Lambda)} 
  \leq Ch \| w_{N,h}-w_\B \|_{L^2(\Wast; \tilde{H}^1_\alpha(\Omega_H^\Lambda))}
\]
holds for all $g \in L^2(\Wast; L^2(\Omega_H^\Lambda))$ with norm equal to one. 
In consequence, Theorem~\ref{th:regulH2} and~\eqref{eq:aux494} imply that 
\begin{align*}
  \| w_{N,h}-w_\B \|_{L^2(\Wast\times\Omega_H^\Lambda)} 
  & \leq Ch \| w_{N,h}-w_\B \|_{L^2(\Wast; \tilde{H}^1_\alpha(\Omega_H^\Lambda))} \\
  & \leq C h(N^{-r}+h) \| f \|_{H^r_\0(\Wast; H^{1/2}_\alpha(\Gamma_H^\Lambda))} .
\end{align*}
\end{proof}

\section{Numerical Implementation for Locally Perturbed Surfaces}\label{se:NumMet}
In this section, we describe the numerical implementation of the variational problem~\eqref{eq:var-traBloch3} in detail. For convenience, we solve for the scattered field instead of for the total field and further periodize all quasiperiodic functions, such that the sesqui-linear forms will become $\alpha$-dependent instead of the function spaces. 

Recall that the scattered field $w^s(\alpha,x):=w_\B(\alpha,x)-(\J_\Omega u^i)(\alpha,x)$ belongs to $L^2(\Wast;\, H^1_\alpha(\Omega^\Lambda_H))$ and satisfies the variational problem 
\begin{equation}\label{eq:var-sca}
\begin{split}
\int_\Wast a_{\alpha} (w^s(\alpha,\cdot),v_\B(\alpha,\cdot)) \, d\alpha 
 & + \left[\frac{\Lambda}{2\pi}\right]^{1/2} \int_{\Omega_H^\Lambda} (A_\pert -I)\nabla \big(\J_\Omega^{-1} w^s \big) \cdot \nabla \big(\overline{\J_\Omega^{-1} v_\B} \big)\, dx \\
& - k^2 \left[\frac{\Lambda}{2\pi}\right]^{1/2}\int_{\Omega_H^\Lambda} (c_\pert -1) \J_\Omega^{-1} w^s \, \overline{\J_\Omega^{-1} v_\B} \, dx 
  = 0
\end{split}
\end{equation}
for all $v_\B \in L^2(\Wast;\, \tilde{H}^1_\alpha(\Omega^\Lambda_H))$, together with the variationally formulated boundary conditions
\begin{equation}\label{eq:var-bc}
\int_\Wast \int_\Gamma \left[ w^s(\alpha,\cdot) - (\J_\Omega u^i)(\alpha,\cdot) \right] \overline{t}(\alpha,\cdot) \, dx \, d\alpha = 0 \quad \text{ for all } t \in L^2(\Wast;\, H^{-1/2}_\alpha(\Gamma)).
\end{equation}
We next periodize all functions involved in the latter formulation, that is, we introduce 
\begin{equation*}
  w_\Null(\alpha,x)=e^{\i\alpha x_1}  w^s(\alpha,x)
  \quad \text{ and } \quad 
  v_\Null(\alpha,x)=e^{\i\alpha x_1} v_\B(\alpha,x) 
  \quad 
  \text{for $(\alpha,x) \in \Wast\times\Omega_H$,}
\end{equation*}
such that $w_\Null$ and $v_\Null$ in $L^2(\Wast; H^1_0(\Omega^\Lambda_H))$ are for fixed $\alpha$ two $\Lambda$-periodic functions in $x_1$.  
Further, $t_\Null(\alpha,x)=\exp(\i\alpha x_1)t(\alpha,x)$ belongs to $L^2(\Wast; H^{-1/2}_0(\Gamma))$.  

As the gradient $\nabla_x w^s(\alpha,\cdot)$ transforms to {$(\nabla_x + \i\alpha{\bm e_1})(e^{-\i\alpha x_1}w_\Null)$}, the variational problem \eqref{eq:var-sca} for $w^s$ equivalently reformulates for $w_\Null$ as 
\begin{equation}
\label{eq:var-traBloch1}
\begin{aligned}
& \int_\Wast \left[ \int_{\Omega_H^\Lambda} \left[ (\nabla_x + \i\alpha{\bm e_1})w_\Null \cdot (\nabla_x - \i\alpha{\bm e_1})\overline{v_\Null} - k^2 w_\Null\overline{v_\Null}\right] dx - \int_{\Gamma^\Lambda_H}\tilde{T}_\alpha^+(w_\Null|_{\Gamma^\Lambda_H})\overline{v_\Null} \, ds \right] d\alpha \\
& \qquad + \left[\frac{\Lambda}{2\pi}\right]^{1/2} \int_{\Omega_H^\Lambda} (A_\pert -I)\nabla \J_\Omega^{-1} (\exp(-\i\alpha \, \cdot) w_\Null) \cdot \overline{\J_\Omega^{-1} \nabla(\exp(-\i\alpha \, \cdot) v_\Null)} \, dx \\
& \qquad \qquad - k^2 \left[\frac{\Lambda}{2\pi}\right]^{1/2} \int_{\Omega_H^\Lambda} (c_\pert -1) \J_\Omega^{-1} (\exp(-\i\alpha \, \cdot) w_\Null) \, \overline{\J_\Omega^{-1} (-\exp(\i\alpha \, \cdot) v_\Null)} \,dx 
=0
\end{aligned}
\end{equation}
for all $v_\Null \in L^2(\Wast;\tilde{H}^1_0(\Omega_H^\Lambda))$ and 
\begin{equation}\label{eq:aux811}
\int_\Wast \int_{\Gamma^\Lambda}\left[ w_\Null(\alpha,\cdot) - e^{\i\alpha x_1}(\J_\Omega u^i)(\alpha,\cdot) \right] \overline{t_\Null}(\alpha,x)\, ds \, d\alpha = 0
\end{equation}
for all $t\in L^2(\Wast; \, H^{-1/2}_0(\Gamma^\Lambda))$. 
In~\eqref{eq:var-traBloch1}, the modified Dirichlet-to-Neumann map $\tilde{T}^+_\alpha$ is defined on periodic functions $\varphi = \sum_{j\in\Z} \hat{\varphi}(j) \exp(\i j\Lambda  x_1)$ in $H^s(\Gamma_H^\Lambda)$ by 
\begin{equation}\label{eq:perioDtNmodify}
   \tilde{T}_\alpha^+ \left( \varphi \right)\Big|_{\Gamma_H^\Lambda}
   = \bigg[\i \sum_{j\in\Z} \sqrt{k^2 - |\Lambda^\ast j - \alpha|^2} \, \hat{\varphi}(j) \, e^{\i j\Lambda   x_1} \bigg]\bigg|_{\Gamma_H^\Lambda}.
\end{equation}
For simplicity, we introduce short-hand notation for~\eqref{eq:perioDtNmodify}, writing  
\begin{align} \label{eq:aux699}
  b'(w,v) = 
  & \left[\frac{\Lambda}{2\pi}\right]^{1/2} \int_{\Omega_H^\Lambda} \left[(A_\pert -I)\nabla \J_\Omega^{-1} ( \exp(-\i\alpha \, \cdot)w) \cdot \nabla \overline{\J_\Omega^{-1}(\exp(-\i\alpha \, \cdot) v)} \right. \\
  & \qquad\qquad\qquad - \left. k^2(c_\pert -1) \J_\Omega^{-1} ( \exp(-\i\alpha \, \cdot)w) \, \overline{\J_\Omega^{-1}(\exp(-\i\alpha \, \cdot) v)} \right] dx \nonumber
\end{align}
for arbitrary $w,v \in H^1(\Omega_H^\Lambda)$ and abbreviate the term inside the $\alpha$-integral in the first line in~\eqref{eq:var-traBloch1} by $a_{\alpha}'(w_\Null,v_\Null)$. 
Then~\eqref{eq:var-traBloch1} reads 
\begin{equation}\label{eq:aux672}
\int_\Wast a_{\alpha}'(w_\Null(\alpha,\cdot), v_\Null(\alpha,\cdot)) \, d\alpha + b'(w_\Null,v_\Null) = 0
\quad \text{for all } v_\Null \in \tilde{X}_{N,h}. 
\end{equation}
Let us emphasize that $b'$ implements the coupling due to the perturbation of the periodic surface between the different quasiperiodic components of the Bloch transformed solution. 

We next discretize the latter family of problems by finite elements and recall from the definition of the finite-dimensional approximation space $\tilde{X}_{N,h}$ in \eqref{eq:XN} the piecewise constant set of functions $\{\psi_{N}^{(j)} \}_{j=1}^N$ in $\alpha$, as well as the piecewise linear and globally continuous nodal basis $\{\phi_{M}^{(l)}\}_{l=1}^{M}$ of the finite-dimensional approximation space $\tilde{V}_h \subset \tilde{H}^1_0(\Omega_H^\Lambda)$ of basis functions that vanish on $\Gamma^\Lambda$. 
We now introduce a larger approximation space $V_h = \mathrm{span}\{ \phi_{M'}^{(l)} \}_{l=1}^{M'}$ with $M' =M'(h) \geq M$ spanned by \textit{all} basis functions defined on the mesh, i.e., also those that do not vanish on $\Gamma^\Lambda$, such that $\phi_{M'}^{(l)} = \phi_{M}^{(l)}$ for $1 \leq l \leq M$ and 
\[
  V_h := \mathrm{span}\{\phi_{M'}^{(l)}\}_{l=1}^{M'} \subset H^1_0(\Omega_H^\Lambda).
\]
If we denote the nodes of the mesh defining $V_h$ by $x_{M'}^{(1)}, \dots, x_{M'}^{(M')}$, then $\phi_{M'}^{(m)}$ is piecewise linear on each triangle of the mesh and satisfies $\phi_{M'}^{(\ell)} (x_{M'}^{(m)}) = \delta_{\ell,m}$ for $1\leq \ell,m \leq M'$.  
In particular, the mesh functions $\phi_{M'}^{(M+1)}, \dots, \phi_{M'}^{(M')}$ are linked to the nodes $x_{M'}^{(M+1)}, \dots, x_{M'}^{(M')}$ that are contained by  $\Gamma^\Lambda$; these basis functions hence yield the boundary values of a function in $V_h$. 
This finite element space $V_h$ then defines 
\begin{equation}
  X_{N,h} = \left\{v_{N,h}(\alpha,x) = e^{-\i\alpha x_1} \sum_{j=1}^N\sum_{l=1}^{M'} v_{N,h}^{(j,l)} \psi_{N}^{(j)}(\alpha) \phi_{M'}^{(l)}(x): \,v_{N,h}^{(j,l)}\in\C\right\} \subset L^2(\Wast; H^1_\alpha(\Omega_H^\Lambda))
\end{equation}
as well as subspaces $Y_{N,h}^{(j)}$ of functions that are constant in $\alpha \in (\alpha_N^{(j)}-\pi/(N\Lambda), \alpha_N^{(j)}+\pi/(N\Lambda)]$, 
\begin{equation}
  Y_{N,h}^{(j)} = \left\{v_{N,h}^{(j)}(\alpha,x) = \sum_{l=1}^{M'} v_{N,h}^{(j,l)}\psi_{N}^{(j)}(\alpha) \phi_{M'}^{(l)}(x): \,v_{N,h}^{(j,l)}\in\C\right\} \subset L^2(\Wast; H^1_0(\Omega_H^\Lambda))
\end{equation}
for $j=1,\dots,N$, and a corresponding subspace of functions that vanish on $\Gamma^\Lambda$, 
\begin{equation} \hspace*{-2mm}
\tilde{Y}_{N,h}^{(j)} =\left\{v_{N,h}^{(j)}(\alpha,x) = \sum_{l=1}^{{M'}} v_{N,h}^{(j,l)}\psi_{N}^{(j)}(\alpha)\phi_{{M'}}^{(l)}(x):\,v_{N,h}^{(j,l)}\in\C, \, v_{N,h}^{(j,l)} = 0 \text{ if } \left. \phi_{M'}^{(l)} \right|_{\Gamma^\Lambda} \not \equiv 0 \right\}. 
\end{equation}
Thus, setting  
\begin{equation}
Y^0_{N,h}=Y_{N,h}^{(1)} \oplus\cdots \oplus Y_{N,h}^{(N)}
\quad\text{ and } \quad 
\tilde{Y}^0_{N,h}=\tilde{Y}_{N,h}^{(1)} \oplus\cdots \oplus\tilde{Y}_{N,h}^{(N)},
\end{equation}
we note that the solution $w_\Null \in X_{N,h}$ to~\eqref{eq:aux672} can be represented by a unique element $(w_\Null^{(j)})_{j=1}^N$ in $Y_{N,h}^0$ as $w_\Null = \exp(-\i \alpha (\cdot)_1)\sum_{j=1}^N w_\Null^{(j)}$. 
%
%
(Here, $(\cdot)_1$ denotes the first component of the function's argument.)
The inverse Bloch transform $\J_\Omega^{-1}$ applied to $\exp(-\i\alpha (\cdot)_1) w_\Null$ that implicitly appears in~\eqref{eq:aux672} hence equals the numerical inverse Bloch transform $\J_{\Omega,N}^{-1}$ from~\eqref{eq:aux462}, applied to $\big\{ \exp(-\i \alpha_N^{(j)} (\cdot)_1) w_\Null^{(j)}(\alpha_N^{(j)}, \cdot) \big\}_{j=1}^N$, 
\[
  \J_\Omega^{-1} \left( \exp(-\i\alpha (\cdot)_1) w_\Null \right) 
  = \J_{\Omega,N}^{-1} \left( \{ w_{N,h}(\alpha_N^{(j)},\cdot) \}_{j=1}^N \right) 
  = \J_{\Omega,N}^{-1} \left( \{ e^{-\i \alpha_N^{(j)} (\cdot)_1} w_\Null^{(j)}(\alpha_N^{(j)},\cdot) \}_{j=1}^N \right). 
\]
The latter equation actually shows via~\eqref{eq:aux462} how we implement the Bloch transforms in the form $b'$ from~\eqref{eq:aux672} in our numerical examples. 

The Galerkin discretization~\eqref{eq:aux672} can now be reformulated via the tuple $(w_\Null^{(j)})_{j=1}^N \in Y^0_{N,h}$ just introduced, 
\begin{align} \label{eq:aux709}
 & \int_{\alpha_N^{(j)}-\pi/(N\Lambda)}^{\alpha_N^{(j)}+\pi/(N\Lambda)} a_{\alpha}'(w_\Null^{(j)}, \, v_\Null^{(j)}) \,d\alpha 
 + b'\big( \{ w_\Null^{(\ell)} \}_\ell, \, \{ v_\Null^{(\ell)}\}_\ell \big) = 0
\end{align}
for all tuples $(v_\Null^{(1)}, \dots, v_\Null^{(N)}) \in \tilde{Y}^0_{N,h}$ and $j=1,\dots,N$.
The boundary conditions~\eqref{eq:aux811} on $\Gamma^\Lambda$ can be implemented directly at the nodal points of the finite elements: 
To this end, consider all nodal points $x_{M'}^{(M+1)}, \dots, x_{M'}^{(M')}$ on $\Gamma^\Lambda$ and denote the corresponding basis function of $V_h$ by {$\phi_{M'}^{(m)}$, $m=M+1, \dots, M'$}.
The first coordinate of these points hence is $(x_{M'}^{(\ell)})_1$, such that we impose that  
\begin{equation}\label{eq:side}
  \frac{2\pi}{N\Lambda} w_\Null^{(j)}(x_{M'}^{(\ell)}) 
  = \int_{\alpha_N^{(j)}-\pi/(N\Lambda)}^{\alpha_N^{(j)}+\pi/(N\Lambda)} \hspace*{-1.8mm} e^{\i\alpha \big(x_{M'}^{(\ell)}\big)_1}(\J_{\Omega}u^i)(\alpha,x_{M'}^{(\ell)}) \,d\alpha, \quad 
  1\leq j \leq N, \, M+1 \leq \ell \leq M'.
\end{equation} 
For notational simplicity, let us now identify the function $v_\Null^{(j)} \in \tilde{Y}_{N,h}^{(j)}$, which is by construction supported in $[\alpha_N^{(j)} - \pi/(N\Lambda), \alpha_N^{(j)} + \pi/(N\Lambda)] \times \overline{\Omega_H^\Lambda}$, with the element $(0, \dots, v_\Null^{(j)}, \dots 0)$ in $\tilde{Y}_{N,h}^{0}$, and further set $\tilde{a_j}(w, \, v_\Null^{(j)})=\int_{\alpha_N^{(j)}-\pi/(N\Lambda)}^{\alpha_N^{(j)}+\pi/(N\Lambda)} a'_\alpha(w,\,v_\Null^{(j)})\,d\alpha$ as well as   
\begin{equation*}
\tilde{b_j}( \{ w_\Null^{(\ell)} \}_\ell, \, v_\Null^{(j)}) 
= b'\big( \{ w_\Null^{(\ell)} \}_\ell, \,  v_\Null^{(j)} \big) 
\qquad \text{for $j=1,\dots,N$.}
\end{equation*}
Then we get the following discrete variational problem for $(w_\Null^{(j)})_{j=1}^N \in Y_{N,h}^0$ that is equivalent to~\eqref{eq:aux709}, 
\begin{equation*}
  \tilde{a_j}(w_\Null^{(j)}, \, v_\Null^{(j)}) 
  + \tilde{b_j}\big( \{ w_\Null^{(\ell)} \}_\ell, \, v_\Null^{(j)} \big) 
  = 0
  \qquad \text{for $j=1,\dots,N$}
\end{equation*}
and all $(v_0^{(1)}, \dots, v_0^{(N)}) \in \tilde{Y}_{N,h}^0$, together with the constraints
\begin{equation}
 w_\Null^{(j)}(x_{M'}^{(\ell)}) 
 = \frac{N\Lambda}{2\pi}\int_{\alpha_N^{(j)}-\pi/(N\Lambda)}^{\alpha_N^{(j)}+\pi/(N\Lambda)} e^{\i\alpha \big( x_{M'}^{(\ell)} \big)_1}(\J_{\Omega}u^i)(\alpha,x_{M'}^{(\ell)})\,d\alpha 
 =: c_{N,h}^{(j)}(x_{M'}^{(\ell)}) 
\end{equation} 
for $j=1,\dots,N$ and $\ell=M+1,\dots, M'$.
The last equation is due to our choice of the finite element space $V_h$ equivalent to~\eqref{eq:side}. 
%
%
%

Numerically, we actually solve a slightly restructured linear system that relies on a further unknown ${u_h = \sum_{j=1}^{M'} u_{h}^{(j)} \phi_{M'}^{(j)}} \in V_h \subset H^1_0(\Omega_H^\Lambda)$ that equals the (discrete) inverse Bloch transform of $(w_\Null^{(1)}, \dots, w_\Null^{(N)}) \in Y_{N,h}^{0}$. In our discretization, the constraint 
\begin{equation}
\label{eq:invBloch}
u_h (x) 
= \J_{\Omega,N}^{-1} (\{ \exp(-\i \alpha_N^{(j)} \, (\cdot)) w_\Null^{(j)} \}_j) 
\stackrel{\eqref{eq:aux462}}{=} \left[ \frac{\Lambda}{2\pi} \right]^{1/2} \sum_{j=1}^N g_N^{(j)} (x_1) \, e^{-\i \alpha_N^{(j)} x_1} w_\Null^{(j)}(x)  
\end{equation}
is added to the linear system as a constraint for $(w_\Null^{(j)})_{j=1}^N = \big(\sum_{m=1}^{M'} w_{N,h}^{(j,m)} \phi_{M'}^{(m)} \big)_{j=1}^N \in Y_{N,h}^0$. 
(The weight functions $g_N^{(j)}$ have been explicitly defined in~\eqref{eq:aux462}.) 
To this end, recall the $M'$ nodal points $x_{M'}^{(1)}, \dots, x_{M'}^{(M)}$ of the mesh defining the periodic finite element space $V_h$ that are not on $\Gamma^\Lambda$, as well as the $M'-M$ nodes $x_{M'}^{(M+1)}, \dots, x_{M'}^{(M')}$ on $\Gamma^\Lambda$. 
Also recall the basis functions $\phi_{M'}^{(\ell)}$ of $V_h$ linked to these nodal points by the conditions  $\phi_{M'}^{(\ell)}(x_{M'}^{(m)}) = \delta_{\ell,m}$. 
Abbreviating $\phi_m = \phi_{M'}^{(m)}$, the discrete solution $W = (w_\Null^{(1)}, \dots, w_\Null^{(N)}) \in Y_{N,h}^0$ with $w_\Null^{(j)} = \sum_{m} w_{N,h}^{(j,m)} \psi_{N}^{(j)} \phi_{m}$ solves the linear system 
\begin{align}
  \sum_{m=1}^M \tilde{a_j}(\phi_m,\phi_\ell) w_{N,h}^{(\ell,m)} + \sum_{m=1}^M \tilde{b_j}( \{ \phi_n \}_n),  \phi_\ell) u_{h}^{(m)} 
    & = 0 \quad \text{for } \ell=1,\dots,M, \ j=1,\dots,N, \nonumber \\ 
  w_{N,h}^{(\ell,m)} = c_{N,h}^{(j)}(x_M^{(\ell)})
    &  \quad \text{for } \ell=M+1,\dots,M' \ j=1,\dots,N, \nonumber \\    
  u_h^{(m)} - \left[ \frac{\Lambda}{2\pi} \right]^{1/2} \sum_{j=1}^N g_N^{(j)}\big((x_{M'}^{(m)})_1 \big) \, e^{-\i \alpha_N^{(j)} (x_{M'}^{(m)})_1} w_{N,h}^{(j,m)} & =0 \quad \text{for } m=1, \dots, M'.  \label{eq:sys}
\end{align}
If we introduce vectors $U=(u_h^{(1)},\dots,u_h^{(M')})^\top$, $W_j=(w_{N,h}^{(j,1)},\dots,w_{N,h}^{(j,M')})^\top$, and $F_j = (F_j^{(\ell)})_{\ell=1}^{M'}$ where $F_j^{(\ell)}=0$ for $\ell=1,\dots,M$ and $F_j^{\ell}=c_{N,h}^{(j)}(x_M^{(\ell)})$ for $\ell=M+1,\dots,M'$, then \eqref{eq:sys} is equivalent to the quadratic matrix-vector equation 
\begin{equation}
\label{eq:FinalMatrix}
\left(
\begin{matrix}
A_1 & 0 & \cdots & 0 & C_1\\
0 & A_2 & \cdots & 0 & C_2\\
\vdots & \vdots & \vdots & \vdots & \vdots \\
0 & 0 & \cdots & A_N & C_N\\
B_1 & B_2 & \cdots & B_N & I_M 
\end{matrix}
\right)
\left(
\begin{matrix}
W_1\\W_2\\ \vdots \\ W_N\\ U
\end{matrix}
\right)=
\left(
\begin{matrix}
F_1\\F_2\\ \vdots \\F_N\\0
\end{matrix}
\right) \in \C^{(N+1)M'}, 
\end{equation}
{with complex $M'\times M'$-matrices $A_j$  and $C_j$ defined by $A_j(m,l) = \tilde{a_j}(\phi_m,\phi_\ell))$ for $1\leq m \leq M,1\leq \ell \leq M'$ and $A_j(m,\ell) = \delta_{m,\ell}$ else}, as well as $C_j(m,l) = \tilde{b_j}(\phi_m,\phi_\ell)$ for $1 \leq m,\ell \leq M$ and $C_j(m,l) = 0$ else. 
Further, 
\begin{equation*}
B_j = -\left[ \frac{2\pi}{\Lambda}\right]^{1/2}\frac{1}{N}\, \mathrm{diag}\bigg[ g_N^{(j)}\big( \big( x_{M'}^{(1)} \big)_1 \big) e^{-\i\alpha_N^{(j)} \big( x_{M'}^{(1)} \big)_1}, \, \dots, \, g_N^{(j)} \big(\big( x_{M'}^{(M')} \big)_1\big) e^{-\i\alpha_N^{(j)} \big( x_{M'}^{(M')} \big)_1}\bigg]
\end{equation*}
for $j=1,\dots, N$.
Of course, after solving this linear system, we have already computed the discrete inverse Bloch transform $u_h$ of the individual solutions $(w_{N,h}^{(1)}, \dots, w_{N,h}^{(N)})$ via $U$. Multiplying $u_h \in V_h \subset H^1_0(\Omega_H^\Lambda)$ by $\exp(-\i\alpha x_1)$ yields an approximation in $\Omega_H^\Lambda$ to the transformed solution $u_\T = u \circ \Phi_\pert$ to the surface scattering problem from the locally perturbed periodic surface $\Gamma_\pert$ we considered originally. Consequently, $u$ itself can on $\Omega_H^\Lambda$ be approximated via the formula $u = u_\T \circ \Phi_\pert^{-1}$. 

When solving the large linear system \eqref{eq:FinalMatrix} of size $(N+1)M' \times (N+1)M'$, one needs to employ an iterative method, as direct solvers become inefficient (at least without parallelization). We chose the GMRES iteration as solution method and pre-condition it in two steps: 
\begin{enumerate}
\item[(1)] Construct an incomplete LU decomposition $(L_j,U_j)$ for each $M'\times M'$-matrix $A_j$, for $n=1,2,\dots,N$. Define the lower triangular matrix $L={\rm diag}(L_1,\dots,L_N,I_M)$ and the upper triangular matrix $U={\rm diag}(U_1,\dots,U_N,I_M)$.
\item[(2)] Use the GMRES iteration procedure with $(L,U)$ be the pre-conditioner with a tolerance $\epsilon>0$ that we typically choose to be $\epsilon = 10^{-6}$. 
\end{enumerate}
This choice certainly is somewhat preliminary as we did not perform large-scale tests agains other preconditioners, and in particular not against parallelized solvers, to tackle~\eqref{eq:FinalMatrix}. 

\section{Numerical Examples}\label{se:NumEg}
In this section, we give some numerical results for the above-presented Bloch transform based method, discretized in~\eqref{eq:FinalMatrix}, together with error estimates and computation times to indicate efficiency. 
We always choose the incident field as the half-space Dirichlet's Green's function
\begin{equation*}
u^i(x):=G(x,y)=\frac{\i}{4}\left[H^{(1)}_0(k|x-y|)-H^{(1)}_0(k|x-y'|)\right], \quad x \not = y \in \R^2_+,
\end{equation*}
where $y'=(y_1,-y_2)^{\top}$ is the mirror point of the source point $y \in \R^2_+ := \{ \tilde{y} \in \R^2: \, \tilde{y}_2 > 0 \}$. 
The source point $y \in \R^2_+$ is in all experiments located below both the periodic and the locally perturbed surfaces on the one hand, and the $x_1$-axis on the other hand. 
This artificial scattering problem then possesses $G(\cdot,y)$ as an explicit solution, which makes the explicit computation of the error of the resulting solution very simple.

Recall that $\mathrm{sinc}$ is the smooth function defined by $\mathrm{sinc}(t)= \sin(t)/t$ when $t\neq 0$ and $\mathrm{sinc}(0)=1$, fix the period $\Lambda$ as $2\pi$, set $\alpha(j)=j+\alpha \in \R$ and 
\begin{equation*}
\beta(j)=\begin{cases}
\sqrt{k^2-\alpha(j)^2} & \text{ if } \alpha(j) \leq k,\\
\i\sqrt{\alpha(j)^2-k^2} & \text{ if } \alpha(j) > k,
\end{cases}
\end{equation*}
such that the incident field has the form
\begin{equation}
\left(\J_{\R}u^ i\right)(\alpha,x)=\frac{1}{2\pi}\sum_{j\in\Z}e^{\i \alpha(j)(x_1-y_1)+i\beta(j)x_2} \mathrm{sinc}(\beta(j)y_2) \, y_2 \qquad \text{for } x_2 > y_2.
\end{equation}
We give the numerical results for two different periodic surfaces given by
\[
f_1(t)=1+\frac{\sin(t)}{4} 
\quad \text{and} \quad 
f_2(t)=1.9+\frac{\sin(t)}{3}-\frac{\cos(2 t)}{4}.
\]
For each surface, two perturbations are considered:
\begin{align*}
g_1(t) &= \exp\left(\frac{1}{t(t+2)}\right)\left(\cos\left[\frac{\pi(t+2)}{2}\right]+1\right) \qquad \text{for } -2\leq t\leq 0 \text{ and $0$ else, and} \\
g_2(t) &= \exp\left(\frac{1}{(t+1)(t-1)}\right)\sin\left[\pi(t+1)\right] \qquad \text{for } -1\leq t\leq 1 \text{ and $0$ else.}
\end{align*}
The surfaces $\Gamma_1$, $\Gamma_2$, $\Gamma_3$ and $\Gamma_4$ are four locally perturbed periodic surfaces defined by 
\begin{eqnarray*}
&&\Gamma_{1,2} =\{(x,f_1(t)+g_{1,2}(t)):\,x\in\R\} 
\quad \text{and} \quad 
\Gamma_{3,4}=\{(x,f_2(t)+g_{1,2}(t)):\,x\in\R\}.
\end{eqnarray*}

\begin{figure}[hhhtttttt]
\centering
\begin{tabular}{c c}
\includegraphics[width=0.45\textwidth]{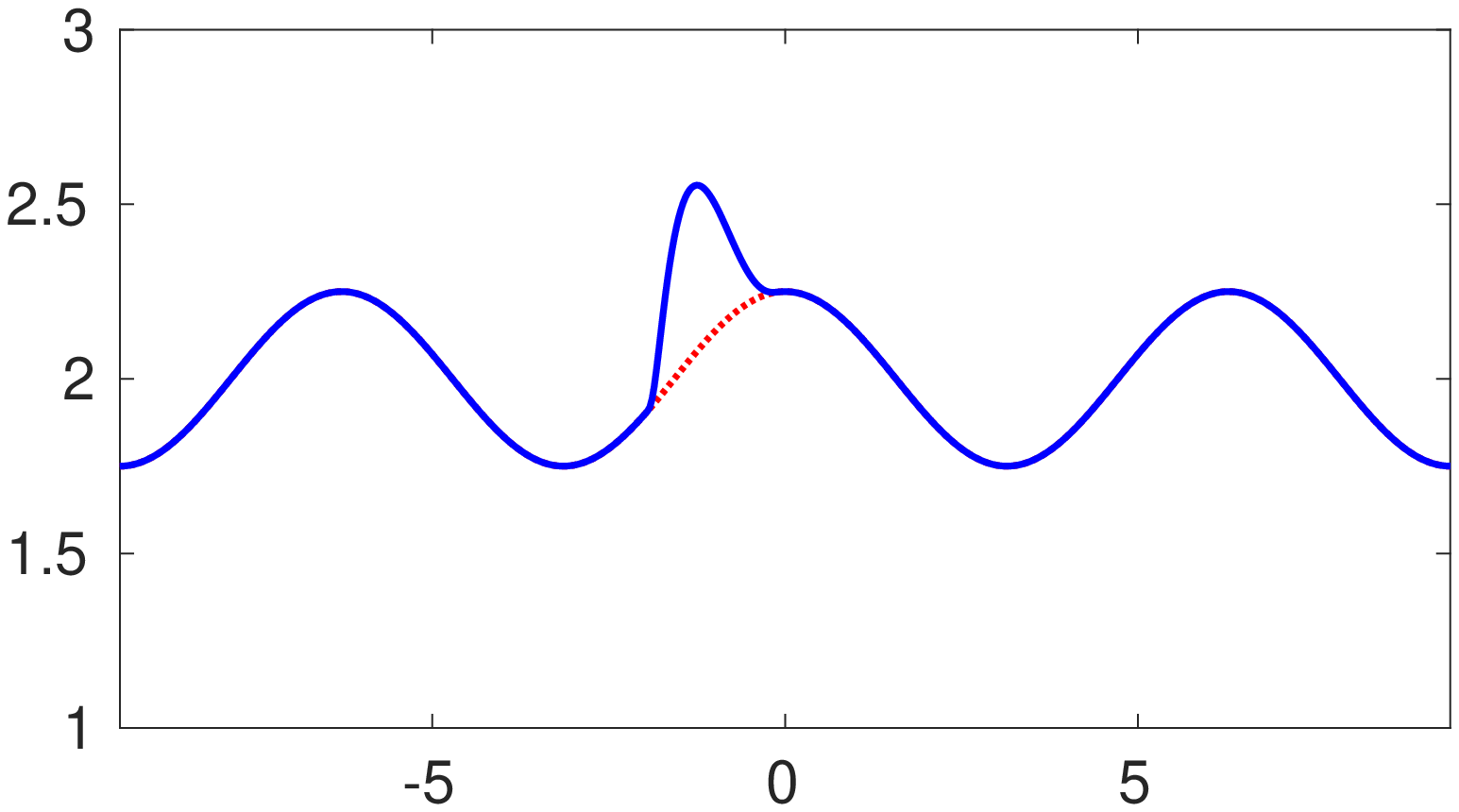}
& 
\includegraphics[width=0.45\textwidth]{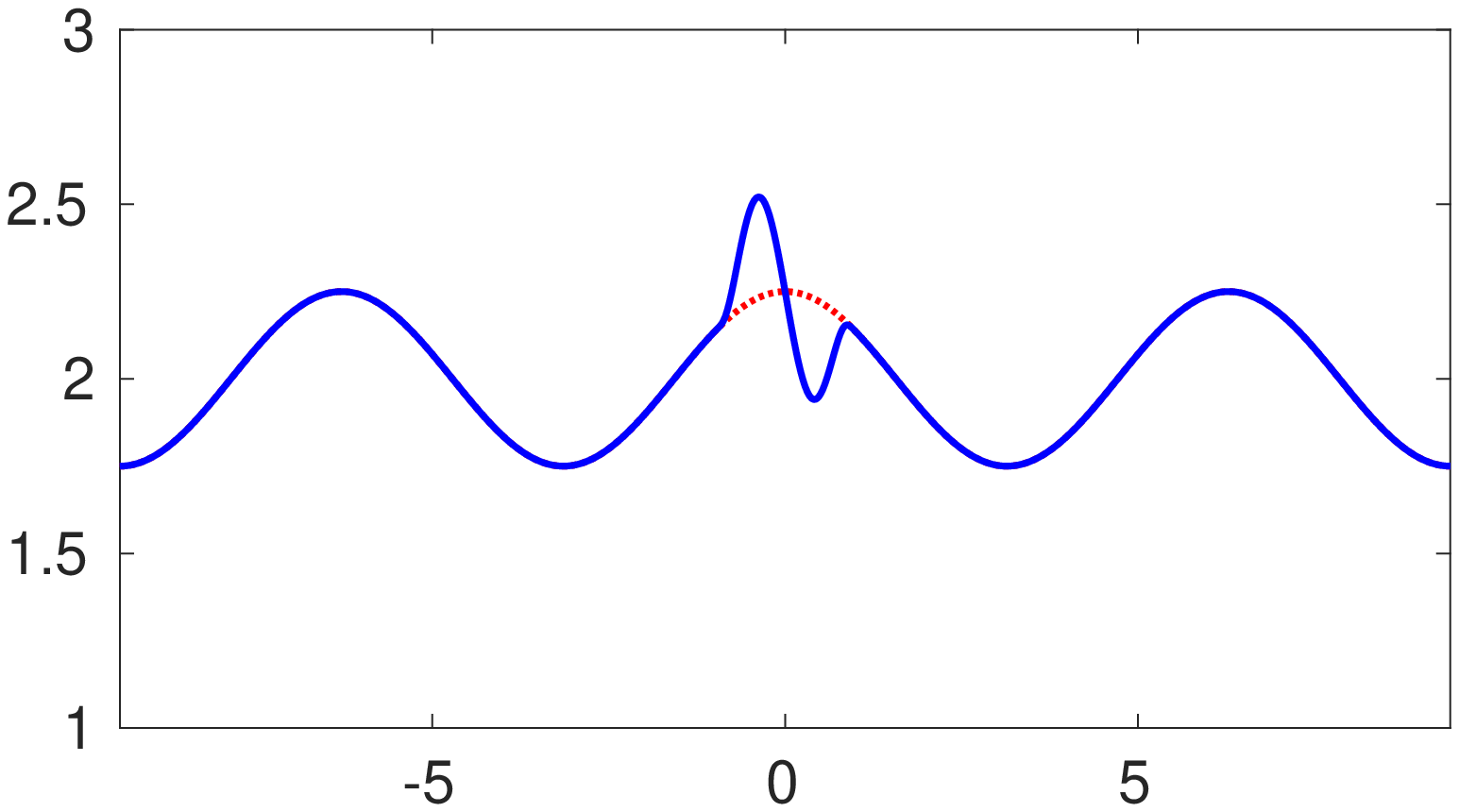}
\\
(a) $f_1+g_1$ & (b) $f_1+g_2$\\
\includegraphics[width=0.45\textwidth]{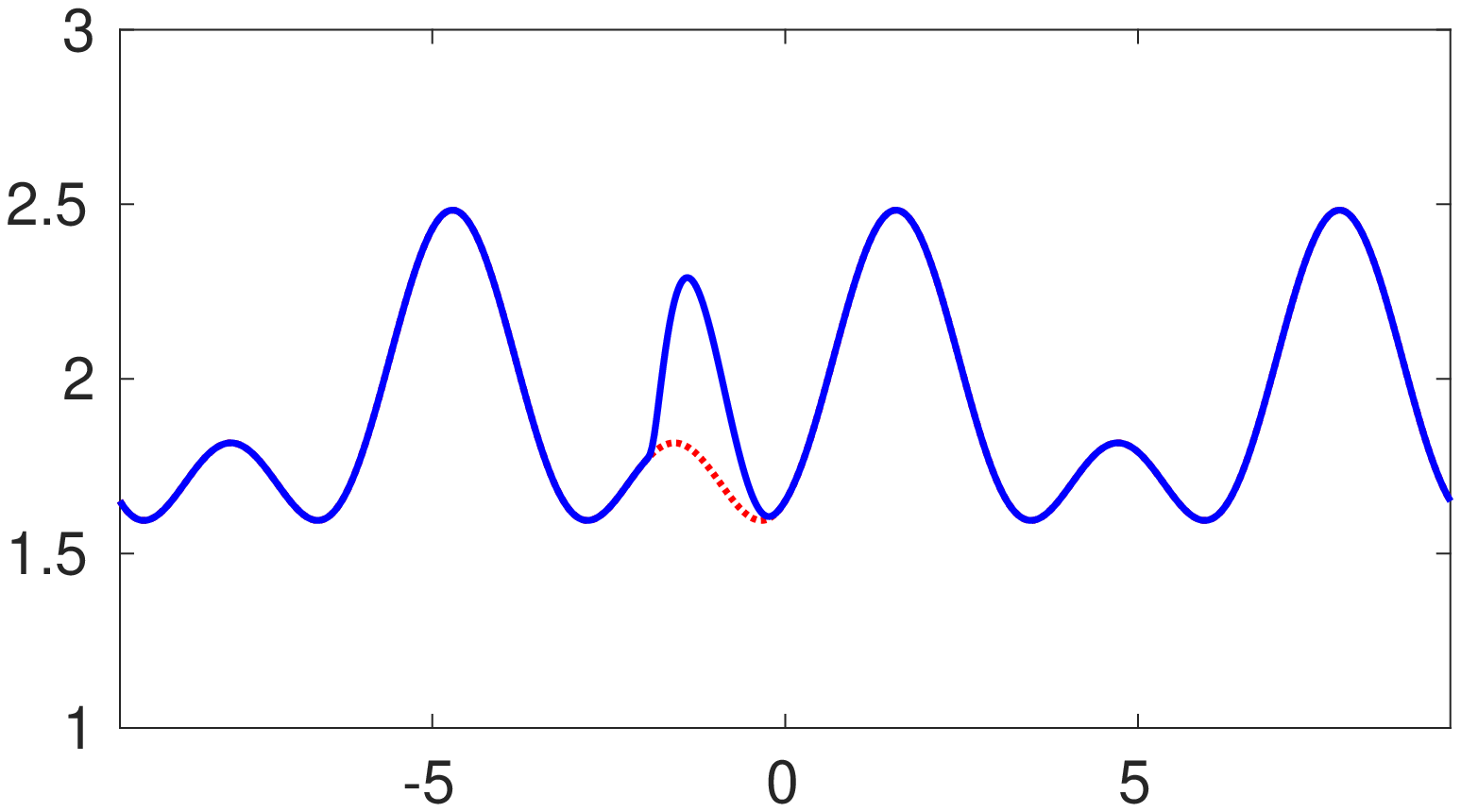}
& 
\includegraphics[width=0.45\textwidth]{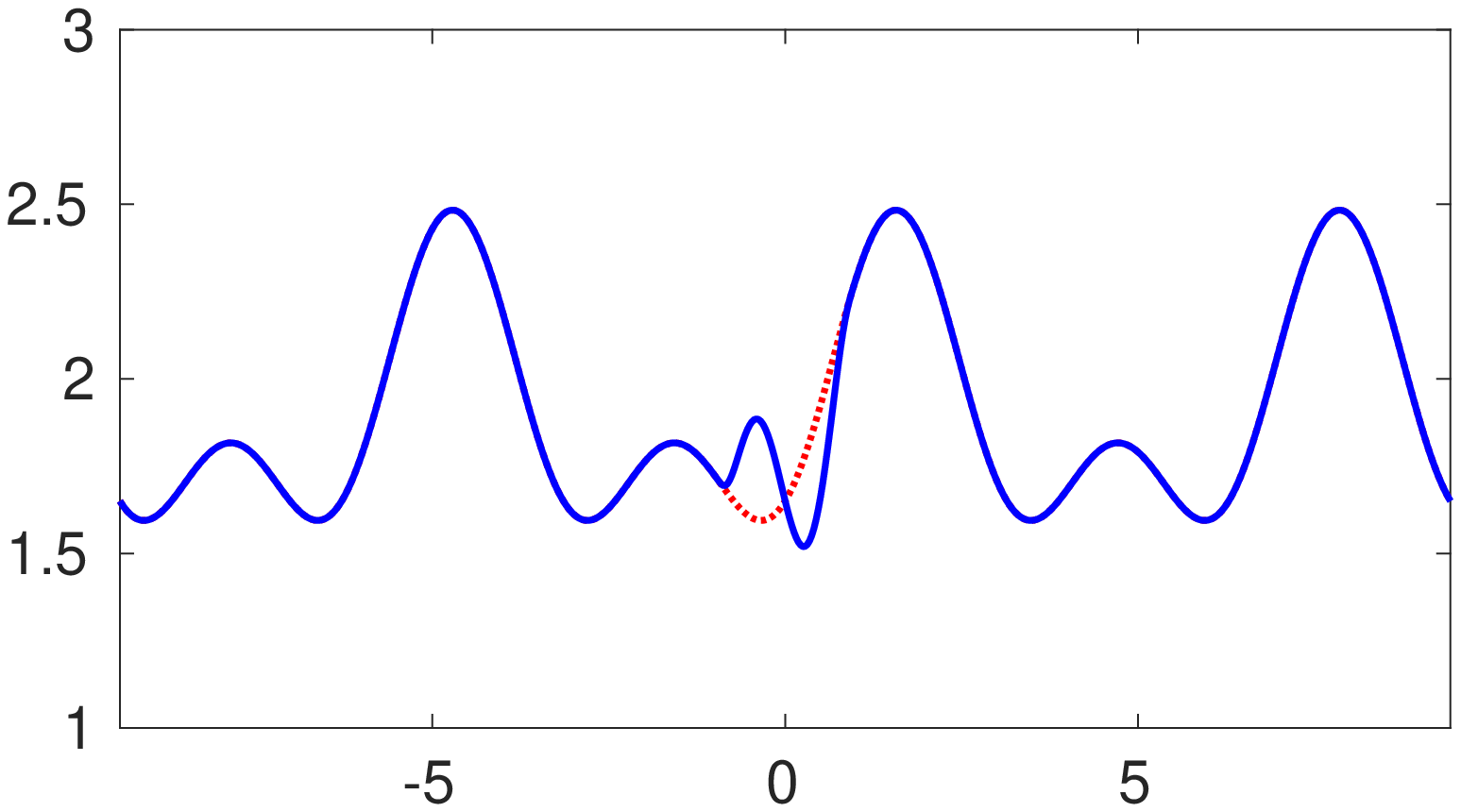}
\\
(c) $f_2+g_1$ & (d) $f_2+g_2$ 
\end{tabular}%
\caption{(a)-(b): The two surfaces $\Gamma_{1,2}$ defined by the unctions $f_1$ and  perturbations $g_1$ and $g_2$; (c)-(d): The two surfaces $\Gamma_{3,4}$ defined by the functions $f_2$ and  perturbations $g_1$ and $g_2$. The dotted lines mark the periodic surfaces and the solid lines mark the locally perturbed ones.}
\end{figure}

For each surface, we chose $H=4$ and evaluated the numerical solutions $u_{N,h}$ on $\Omega^{2\pi}_4$ for two different source points, i.e., $y=(0.5,0.4)^\top$ and $y=(-2,0.2)^\top$ and two wave numbers $k=1$ and $k=10$ by numerically solving the linear system~\eqref{eq:FinalMatrix} by the described preconditioned GMRES algorithm. As mentioned, the exact scattered field equals (minus) the incident field, which allows to compute relative errors $\|u_{N,h}-u\|_{L^2(\Omega^{2\pi}_4)}/\|u\|_{L^2(\Omega^{2\pi}_4)}$.

Table~\ref{surf1p1k1} and Table~\ref{surf1p1k10} show the relative errors for the numerical solutions for the surface $\Gamma_1$, Table~\ref{surf1p2k1} and Table~\ref{surf1p2k10} shows the results for $\Gamma_2$, Table~\ref{surf2p1k1} and Table~\ref{surf2p1k10} show the results for $\Gamma_3$, and the results of $\Gamma_4$ are in Table~\ref{surf2p2k1} and Table~\ref{surf2p2k10}. For each example, the results are computed for mesh sizes $h=0.16, 0.08, 0.04, 0.02, 0.01$ and $N=20, 40, 80,160, 320$.

\begin{table}[hhhtttttt]
\centering
\begin{tabular}
{|p{1.8cm}<{\centering}||p{2cm}<{\centering}|p{2cm}<{\centering}
 |p{2cm}<{\centering}|p{2cm}<{\centering}|p{2cm}<{\centering}|}
\hline
  & $h=0.16$ & $h=0.08$ & $h=0.04$ & $h=0.02$ & $h=0.01$\\
\hline
\hline
$N=20$&$8.55$E$-03$&$7.39$E$-03$&$7.25$E$-03$&$7.22$E$-03$&$7.21$E$-03$\\
\hline
$N=40$&$4.83$E$-03$&$2.90$E$-03$&$2.69$E$-03$&$2.66$E$-03$&$2.65$E$-03$\\
\hline
$N=80$&$3.92$E$-03$&$1.38$E$-03$&$1.03$E$-03$&$1.00$E$-03$&$9.94$E$-04$\\
\hline
$N=160$&$3.73$E$-03$&$9.56$E$-04$&$4.56$E$-04$&$4.11$E$-04$&$4.07$E$-04$\\
\hline
$N=320$&$3.68$E$-03$&$8.57$E$-04$&$2.77$E$-04$&$2.36$E$-04$&$2.32$E$-04$\\
\hline
\end{tabular}
\caption{Relative $L^2$-errors for Example 1 (surface $\Gamma_1$, source at $y=(0.5,0.4)^\top$, $k=1$).}
\label{surf1p1k1}
\end{table}

\begin{table}[hhhtttttt]
\centering
\begin{tabular}
{|p{1.8cm}<{\centering}||p{2cm}<{\centering}|p{2cm}<{\centering}
 |p{2cm}<{\centering}|p{2cm}<{\centering}|p{2cm}<{\centering}|}
\hline
  & $h=0.08$ & $h=0.04$ & $h=0.02$ & $h=0.01$\\
\hline
\hline
$N=20$&$3.79$E$-01$&$1.02$E$-01$&$3.99$E$-02$&$1.35$E$-02$\\
\hline
$N=40$&$3.80$E$-01$&$1.01$E$-01$&$2.84$E$-02$&$7.20$E$-03$\\
\hline
$N=80$&$3.81$E$-01$&$1.01$E$-01$&$2.64$E$-02$&$6.02$E$-03$\\
\hline
$N=160$&$3.81$E$-01$&$1.01$E$-01$&$2.61$E$-02$&$5.82$E$-03$\\
\hline
$N=320$&$3.81$E$-01$&$1.01$E$-01$&$2.32$E$-02$&$5.78$E$-03$\\
\hline
\end{tabular}
\caption{Relative $L^2$-errors for Example 2 (surface $\Gamma_1$, source at $y=(-2,0.2)^\top$, $k=10$).}
\label{surf1p1k10}
\end{table}

\begin{table}[hhhtttttt]
\centering
\begin{tabular}
{|p{1.8cm}<{\centering}||p{2cm}<{\centering}|p{2cm}<{\centering}
 |p{2cm}<{\centering}|p{2cm}<{\centering}|p{2cm}<{\centering}|}
\hline
  & $h=0.16$ & $h=0.08$ & $h=0.04$ & $h=0.02$ & $h=0.01$\\
\hline
\hline
$N=20$&$8.98$E$-03$&$7.40$E$-03$&$7.32$E$-03$&$7.32$E$-03$&$7.32$E$-03$\\
\hline
$N=40$&$5.98$E$-03$&$2.94$E$-03$&$2.69$E$-03$&$2.68$E$-03$&$2.69$E$-03$\\
\hline
$N=80$&$5.50$E$-03$&$1.61$E$-03$&$1.03$E$-03$&$9.67$E$-04$&$9.97$E$-04$\\
\hline
$N=160$&$5.45$E$-03$&$1.34$E$-03$&$5.00$E$-04$&$4.04$E$-04$&$4.00$E$-04$\\
\hline
$N=320$&$5.45$E$-03$&$1.31$E$-03$&$3.76$E$-04$&$2.28$E$-04$&$2.17$E$-04$\\
\hline
\end{tabular}
\caption{Relative $L^2$-errors for Example 3 (surface $\Gamma_2$, source at $y=(0.5,0.4)^\top$, $k=1$).}
\label{surf1p2k1}
\end{table}

\begin{table}[hhhtttttt]
\centering
\begin{tabular}
{|p{1.8cm}<{\centering}||p{2cm}<{\centering}|p{2cm}<{\centering}
 |p{2cm}<{\centering}|p{2cm}<{\centering}|p{2cm}<{\centering}|}
\hline
  & $h=0.08$ & $h=0.04$ & $h=0.02$ & $h=0.01$\\
\hline
\hline
$N=20$&$3.78$E$-01$&$1.03$E$-01$&$2.78$E$-02$&$1.43$E$-02$\\
\hline
$N=40$&$3.77$E$-01$&$1.03$E$-01$&$2.66$E$-02$&$7.20$E$-03$\\
\hline
$N=80$&$3.78$E$-01$&$1.03$E$-01$&$2.65$E$-02$&$7.16$E$-03$\\
\hline
$N=160$&$3.78$E$-01$&$1.03$E$-01$&$2.64$E$-02$&$6.96$E$-03$\\
\hline
$N=320$&$3.78$E$-01$&$1.03$E$-01$&$2.64$E$-02$&$6.92$E$-03$\\
\hline
\end{tabular}
\caption{Relative $L^2$-errors for Example 4 (surface $\Gamma_2$, source at $y=(-2,0.2)^\top$, $k=10$).}
\label{surf1p2k10}
\end{table}

\begin{table}[hhhtttttt]
\centering
\begin{tabular}
{|p{1.8cm}<{\centering}||p{2cm}<{\centering}|p{2cm}<{\centering}
 |p{2cm}<{\centering}|p{2cm}<{\centering}|p{2cm}<{\centering}|}
\hline
  & $h=0.16$ & $h=0.08$ & $h=0.04$ & $h=0.02$ & $h=0.01$\\
\hline
\hline
$N=20$&$8.87$E$-03$&$7.80$E$-03$&$7.69$E$-03$&$7.67$E$-03$&$7.67$E$-03$\\
\hline
$N=40$&$4.91$E$-03$&$3.02$E$-03$&$2.82$E$-03$&$2.80$E$-03$&$2.80$E$-03$\\
\hline
$N=80$&$4.01$E$-03$&$1.45$E$-03$&$1.08$E$-03$&$1.05$E$-03$&$1.04$E$-03$\\
\hline
$N=160$&$3.84$E$-03$&$1.06$E$-03$&$4.75$E$-04$&$4.23$E$-04$&$4.21$E$-04$\\
\hline
$N=320$&$3.81$E$-03$&$9.83$E$-04$&$2.99$E$-04$&$2.21$E$-04$&$2.19$E$-04$\\
\hline
\end{tabular}
\caption{Relative $L^2$-errors for Example 5 (surface $\Gamma_3$, source at $y=(0.5,0.4)^\top$, $k=1$).}
\label{surf2p1k1}
\end{table}

\begin{table}[hhhtttttt]
\centering
\begin{tabular}
{|p{1.8cm}<{\centering}||p{2cm}<{\centering}|p{2cm}<{\centering}
 |p{2cm}<{\centering}|p{2cm}<{\centering}|p{2cm}<{\centering}|}
\hline
  & $h=0.08$ & $h=0.04$ & $h=0.02$ & $h=0.01$\\
\hline
\hline
$N=20$&$3.93$E$-01$&$1.06$E$-01$&$2.88$E$-02$&$1.55$E$-02$\\
\hline
$N=40$&$3.94$E$-01$&$1.06$E$-01$&$2.78$E$-02$&$1.14$E$-02$\\
\hline
$N=80$&$3.94$E$-01$&$1.06$E$-01$&$2.77$E$-02$&$1.08$E$-02$\\
\hline
$N=160$&$3.94$E$-01$&$1.06$E$-01$&$2.77$E$-02$&$1.07$E$-02$\\
\hline
$N=320$&$3.94$E$-01$&$1.06$E$-01$&$2.32$E$-02$&$1.07$E$-02$\\
\hline
\end{tabular}
\caption{Relative $L^2$-errors for Example 6 (surface $\Gamma_3$, source at $y=(-2,0.2)^\top$, $k=10$).}
\label{surf2p1k10}
\end{table}

\begin{table}[hhhtttttt]
\centering
\begin{tabular}
{|p{1.8cm}<{\centering}||p{2cm}<{\centering}|p{2cm}<{\centering}
 |p{2cm}<{\centering}|p{2cm}<{\centering}|p{2cm}<{\centering}|}
\hline
  & $h=0.16$ & $h=0.08$ & $h=0.04$ & $h=0.02$ & $h=0.01$\\
\hline
\hline
$N=20$&$9.23$E$-03$&$7.83$E$-03$&$7.76$E$-03$&$7.76$E$-03$&$7.76$E$-03$\\
\hline
$N=40$&$5.71$E$-03$&$3.06$E$-03$&$2.83$E$-03$&$2.82$E$-03$&$2.83$E$-03$\\
\hline
$N=80$&$5.08$E$-03$&$1.61$E$-03$&$1.07$E$-03$&$1.04$E$-03$&$1.04$E$-03$\\
\hline
$N=160$&$5.01$E$-03$&$1.32$E$-03$&$4.88$E$-04$&$4.03$E$-04$&$4.04$E$-04$\\
\hline
$N=320$&$5.01$E$-03$&$1.29$E$-03$&$3.48$E$-04$&$2.28$E$-04$&$1.93$E$-04$\\
\hline
\end{tabular}
\caption{Relative $L^2$-errors for Example 7 (surface $\Gamma_4$, source at $y=(0.5,0.4)^\top$, $k=1$).}
\label{surf2p2k1}
\end{table}

\begin{table}[hhhtttttt]
\centering
\begin{tabular}
{|p{1.8cm}<{\centering}||p{2cm}<{\centering}|p{2cm}<{\centering}
 |p{2cm}<{\centering}|p{2cm}<{\centering}|p{2cm}<{\centering}|}
\hline
  & $h=0.08$ & $h=0.04$ & $h=0.02$ & $h=0.01$\\
\hline
\hline
$N=20$&$4.36$E$-01$&$1.19$E$-01$&$3.19$E$-02$&$1.64$E$-02$\\
\hline
$N=40$&$4.36$E$-01$&$1.19$E$-01$&$3.09$E$-02$&$1.23$E$-02$\\
\hline
$N=80$&$4.37$E$-01$&$1.19$E$-01$&$3.08$E$-02$&$1.17$E$-02$\\
\hline
$N=160$&$4.37$E$-01$&$1.19$E$-01$&$3.07$E$-02$&$1.15$E$-02$\\
\hline
$N=320$&$4.37$E$-01$&$1.19$E$-01$&$2.64$E$-02$&$1.15$E$-02$\\
\hline
\end{tabular}
\caption{Relative $L^2$-errors for Example 8 (surface $\Gamma_4$, source at $y=(-2,0.2)^\top$, $k=10$).}
\label{surf2p2k10}
\end{table}

As is shown in Tables 1-8, the relative error decreases in $N$ and $h$ up to error stagnation. For wave number $k=1$, the error caused by $N$ is the dominant one, such that the error decrease as $h$ gets small is sometimes comparatively small, see Tables 1, 3, 5, and 7. For $k=10$, this is the exact opposite, see Tables 2, 4, 6, and 8. When $h$ is small enough (see the results for $h=0.01$ in Figures 1, 3, 5, 7), the relative error decreases faster than the rate $O\left(N^{-1}\right)$ than proved theoretically in Theorem~\ref{th:numInvBloch}. In Tables \ref{surf2p1k1t} and \ref{surf2p1k10t}, we also show the computation times of our serial code imlemented in MATLAB for Examples 5 and Example 6 computed on a workstation with an INTEL i7-4790 processor (8 cores at 3.60 GHz) and 32 GB RAM. 
This data excludes the smallest mesh size and the largest discretization of the Brillouin zone $\Wast$ that we merely treated on a comparatively slow workstation with significantly larger memory of 264 GB. 

Finally, we balance the two error terms in the $L^2$-estimate $\|u_{N,h} - u\|_{L^2(\Omega_H^\Lambda)}\leq C [N^{-1}+h^2]$ from~\eqref{eq:errorU} by choosing $h=c_0N^{-1/2}$  for $c_0=2/(5\sqrt{5})$ and $N$ equal to $20,\, 80,\, 320$. The $(N,h)$ pairs are hence $(20,0.04)$, $(80,0.02)$ and $(320,0.01)$. Figure~\ref{fig:rate} shows plots in logarithmic scale of the relative $L^2$-errors for the eight examples from above. The slopes for Examples 1, 3, 5,  and 7 is roughly about $-1.42$, for Example 2, 4, 6, and 8 they are roughly about $-1$. This means that the numerical results converges at the rate of $N^{-1}$ or even faster than shown in Theorem~\ref{th:numInvBloch}.

\begin{figure}[hhhtttttt]
\centering
 \includegraphics[width=0.6\textwidth]{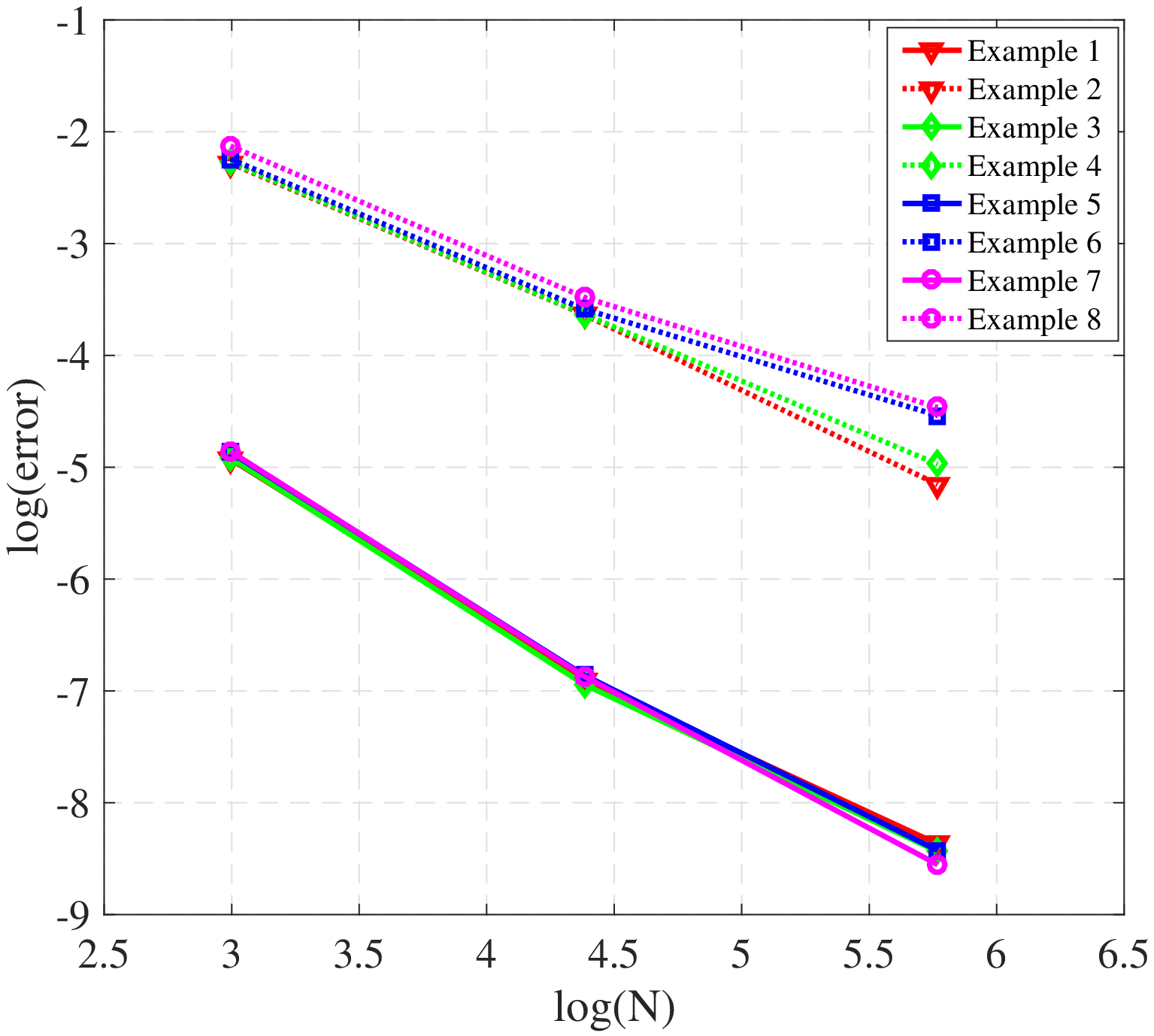}
\caption{The relative $L^2$-errors for the eight considered examples with $h=c_0 N^{-1/2}$ plotted in logarithmic scale over $N$.}
\label{fig:rate}
\end{figure}

\begin{table}[t!!h!b]
\centering
\begin{tabular}
{|p{1.8cm}<{\centering}||p{2cm}<{\centering}|p{2cm}<{\centering}
 |p{2cm}<{\centering}|p{2cm}<{\centering}|p{2cm}<{\centering}|}
\hline
  & $h=0.16$ & $h=0.08$ & $h=0.04$ & $h=0.02$ \\
\hline
\hline
$N=20$& 0.74 & 3.3 & 25 & 277\\
\hline
$N=40$& 1.5 & 7.4 & 55 & 580\\
\hline
$N=80$& 3.7 & 19 & 126 & 1242\\
\hline
$N=160$& 9.3 & 55 & 316 & 2740\\
\hline
\end{tabular}
\caption{Solution time in seconds for Example 3 (surface $\Gamma_3$, source at $y=(0.5,0.4)^\top$, $k=1$).}
\label{surf2p1k1t}
\end{table}

\begin{table}[htb]
\centering
\begin{tabular}
{|p{1.8cm}<{\centering}||p{2cm}<{\centering}|p{2cm}<{\centering}
 |p{2cm}<{\centering}|p{2cm}<{\centering}|p{2cm}<{\centering}|}
\hline
  & $h=0.08$ & $h=0.04$ & $h=0.02$ \\
\hline
\hline
$N=20$& 8.4 & 43 & 341 \\
\hline
$N=40$& 20 & 95 & 691\\
\hline
$N=80$& 46 & 217 & 1500\\
\hline
$N=160$& 120 & 531 & 6524\\
\hline
\end{tabular}
\caption{Solution time in seconds for Example 6 (surface $\Gamma_3$, source at $y=(-2,0.2)^\top$, $k=10$).}\label{surf2p1k10t}
\end{table}

\appendix 

\section{The Floquet-Bloch transform}\label{se:bloch}
We briefly recall mapping properties of the (Floquet-)Bloch transform $\J_\Omega$; standard references on this topic are~\cite{Reed1978} or~\cite{Kuchm1993}, but see also \cite[Annexe B]{Fliss2009}. 
We define that transform on smooth functions $u: \, \overline\Omega \to \C$ with compact support in $\overline\Omega$ by
\begin{equation}
  \label{eq:BlochZeta}
  \J_\Omega u (\alpha,x)
  = \left[ \frac{\Lambda}{2\pi} \right]^{1/2} \sum_{j\in \Z} u\left( \begin{smallmatrix} x_1 +\Lambda j \\ x_2 \end{smallmatrix} \right) e^{\i \, \Lambda j \, \alpha},
  \qquad  x = \left( \begin{smallmatrix} x_1 \\ x_2 \end{smallmatrix} \right) \in \overline\Omega, \, \alpha \in \R.
\end{equation}
(The same transform for functions defined in $\Omega_H$ is denoted by $\J_\Omega$ as well.) 
One easily computes that $\J_\Omega u(\alpha, \cdot)$ is $\alpha$-quasiperiodic with respect to $\Lambda$,
\begin{equation}\label{eq:qpExt}
  \big( \J_\Omega u \big) \big(\alpha, \left( \begin{smallmatrix} x_1 +\Lambda \\ x_2 \end{smallmatrix} \right) \big)
  = \left[\frac{\Lambda}{2\pi}\right]^{1/2} \sum_{j\in \Z} u  \big( \begin{smallmatrix} x_1 + \Lambda(j+1) \\ x_2 \end{smallmatrix} \big) \, e^{-\i \alpha \cdot \Lambda j}
  = e^{-\i  \Lambda \, \alpha} \J_\Omega u(\alpha, x)
\end{equation}
for $x=(x_1,x_2)^\top \in \Omega$. Further, $\J_\Omega u(\cdot, x)$ is for fixed $x$ a Fourier series in $\alpha$ with basis functions $\alpha \mapsto \exp(\i \, \Lambda j \, \alpha)$ that are $\Lambda^\ast = 2\pi/\Lambda$ periodic.
Thus, introducing the unit cell and the Brillouin zone as 
\[
  \W = \bigg(-\frac{\Lambda}{2}, \frac{\Lambda}{2} \bigg]
  \quad \text{and} \quad
  \Wast = \bigg( -\frac{\Lambda^\ast}{2}, \frac{\Lambda^\ast}{2}\bigg]
  = \bigg( -\frac{\pi}{\Lambda}, \frac{\pi}{\Lambda}\bigg]
\]
implies that knowledge of $(\alpha,x) \mapsto \J_\Omega \phi(\alpha, x)$ in $\Wast \times \Omega_\Lambda^H$ (or $\Wast \times \Omega_\Lambda$) defines the Bloch transform $\J_\Omega u(\alpha, x)$ everywhere in $\R \times \Omega_H$ (or $\R \times \Omega$).

This observation reflects in mapping properties of the Bloch transform. To this end, recall the Bessel potential spaces $H^s(\Omega)$ for $s\in \R$, together with their weighted analogues $H^s_r(\R) := \left\{ \phi \in \mathcal{D}'(\R): \, x_1 \mapsto (1+|x_1|^2)^{r/2} \phi(x_1) \in H^s(\R) \right\}$ for $s,r \in \R$, equipped with their natural norms.
The spaces $H^s(\Omega_H)$ are defined analogously and $H^s_r(\Omega_H)$ contains all $u \in \mathcal{D}'(\Omega_H)$ such that $(1+|x_1|^2)^{r/2} u(x) \in H^s(\Omega_H)$.
 
For $s\in\R$ we further rely on the subspace of quasiperiodic functions $H^s_\alpha(W)$ of $\mathcal{D}'(\R, \C)$ (see~\eqref{eq:qpExt} for a definition of $\alpha$-quasiperiodicity). This space contains all $\alpha$-quasiperiodic distributions $\phi$ with finite norm $\| \phi \|_{H^s_\alpha(W)} = ( \sum_{j \in \Z} (1+|j|^2)^s \, |\hat{\phi}(j)|^2 )^{1/2}$ for $\alpha \in \Wast$, where $\hat\phi(j)$ is the $j$th Fourier coefficient of $\phi$ (the dual evaluation between $\phi$ and $\exp(\i (\Lambda^\ast j-\alpha) \, x_1) / | \det \Lambda|^{1/2} $). Functions in these spaces can be represented by their Fourier series, 
\[
  \phi(x_1) 
  = \frac{1}{| \det \Lambda|^{1/2}} \sum_{j\in\Z} \hat{\phi}(j) e^{\i (\Lambda^\ast j-\alpha) \, x_1 }
  \quad \text{ in $H^s_\alpha(W)$.}
\]
The analogous spaces for functions defined in $\Omega_H$ are $H^s_\alpha(\Omega_\Lambda^H) = \{ u \in \mathcal{D}'(\R\times (0,H)): \, u \text{ is $\alpha$-quasiperiodic in $x_1$ and belongs to } H^1(\Omega_H) \}$ with the usual $H^1$-norm on $\Omega_H$. 

Next, we consider all distributions in $\mathcal{D}'(\R \times \Omega_H)$ that are $\Lambda^\ast$-periodic in their first variable $\alpha$ and quasiperiodic in the first component $x_1$ of their second variable $x$, the quasiperiodicity being equal to the first variable.
For integers $r\in \N$ and $s\in \R$, these distributions define norms 
\begin{equation}\label{eq:HrHs}
  \| \psi \|_{H^\ell_\0(\Wast; H^s_\alpha(\Omega_H))}^2
  =  \sum_{\gamma = 1}^\ell \int_\Wast \| \partial^{\gamma}_\alpha \psi(\alpha, \cdot) \|_{H^s_\alpha(\Omega_H)}^2 \d{\alpha},
\end{equation}
and Hilbert space $H^r_\0(\Wast; H^s_\alpha(\Omega_H))$ as set of those distributions with finite $H^\ell_\0(\Wast; H^s_\alpha(\Omega_H))$-norm.
Interpolation in $\ell$ and a duality argument allow to define these spaces for all $\ell \in \R$.

For a regularity result, we actually also require the family of spaces $W^{1,p}_\0(\Wast; \tilde{H}^1_\alpha(\Omega_H^\Lambda))$ for $1 \leq p < \infty$, that are defined by replacing the $H^r$-norm in $\alpha$ by a $W^{1,p}$-norm; the $p$th power of the norm of these spaces equals 
\[
  \| w \|_{W^{1,p}_\0(\Wast; \tilde{H}^1_\alpha(\Omega_H^\Lambda))}^p
  = \int_{\Wast} \left[ \| w(\alpha,\cdot) \|_{\tilde{H}^1_\alpha(\Omega_H^\Lambda))}^p  + \| \partial_\alpha w(\alpha,\cdot) \|_{\tilde{H}^1_\alpha(\Omega_H^\Lambda))}^p \right] \d{\alpha}, 
  \quad 1 \leq p < \infty.
\]

For all spaces introduced so far involving $\Omega_H$, $\Omega_\Lambda^H$, $\Omega$, or $\Omega_\Lambda$ it is convenient to define the closure of smooth functions that vanish in a neighborhood of $\Gamma$ or $\Gamma^\Lambda$ in the norms defined above; the corresponding subspaces are then denoted by $\tilde{H}^s_\alpha(\Omega_\Lambda^H)$, $H^r_\0(\Wast; \tilde{H}^s_\alpha(\Omega_H))$, and so on. 
 
\begin{theorem}\label{th:BlochOmega}
The Bloch transform $\J_\Omega$ extends to an isomorphism between $H^s_r(\Omega_H)$ and $H^r_\0(\Wast; H^s_\alpha(\Omega_H^\Lambda))$ as well as between $\tilde{H}^s_r(\Omega_H)$ and $H^r_\0(\Wast; \tilde{H}^s_\alpha(\Omega_H^\Lambda))$ for all $s,r \in \R$.
Further, $\J_\Omega$ is an isometry for $s=r=0$ with inverse
  \begin{equation}
    \label{eq:JOmegaInverse}
    \left( \J_\Omega^{-1} w \right) \left( \begin{smallmatrix} x_1 + \Lambda j \\ x_2 \end{smallmatrix} \right)
    = \left[ \frac{\Lambda}{2\pi} \right]^{1/2}
    \int_\Wast w(\alpha, x) e^{i \alpha \, \Lambda j} \d{\alpha}
    \qquad \text{for }  x \in \Omega_H^\Lambda.
  \end{equation}
We actually merely consider the inverse Bloch transform in $\Omega_H^\Lambda$, where the exponential factor in the latter integral can be omitted. 
\end{theorem}

Finally, we introduce Sobolev spaces on $\Gamma_H$ and $\Gamma_H^\Lambda = \{  x \in \Gamma_H: \, x_1 \in \W \} \subset \Gamma_H$ by identifying $\Gamma_H$ with $\R$ and $\Gamma_H^\Lambda$ with $\W$.
The resulting spaces are then denoted by $H^s_r(\Gamma_H)$, $H^s_\alpha(\Gamma_H^\Lambda)$, and $H^r_\0(\Wast; H^s_\alpha(\Gamma_H^\Lambda))$ for $s,r\in \R$.
For $s=\pm 1/2$ it is well-known that these are natural trace spaces of volumetric $H^1$-spaces, see~\cite{McLea2000}.

\section*{Acknowledgements.} The second author was supported by the University of Bremen and the European Union FP7 COFUND under grant agreement n$^\circ{}\,$600411.


\end{document}